\documentclass[a4paper]{amsart}
\usepackage{amsfonts,amsmath,amssymb,amscd,amstext,amsbsy,latexsym,url}
\newtheorem{thm}{Theorem}[section]
\newtheorem{cor}[thm]{Corollary}
\newtheorem{lem}[thm]{Lemma}
\newtheorem{prop}[thm]{Proposition}
\theoremstyle{definition}

\theoremstyle{remark}
\newtheorem{rem}[thm]{Remark}
\newtheorem{exm}[thm]{Example}
\numberwithin{equation}{section}
\begin{document}

\title{Classical Ideal Theory of Maximal subrings in Non-commutative Rings}%
\author{Alborz Azarang}%
%\address{Department of Mathematics, Shahid Chamran University, Ahvaz-Iran }%
%\email{a_azarang@sua.ac.ir}%
%\thanks{}%
%\subjclass{}%

\maketitle

\centerline{Department of Mathematics, Faculty of Mathematical Sciences and Computer,}
\centerline{Shahid Chamran University of Ahvaz, Ahvaz-Iran}
\centerline{a${}_{-}$azarang@scu.ac.ir}
\centerline{ORCID ID: orcid.org/0000-0001-9598-2411}
%\date{}%
%\dedicatory{}%
%\commby{}%
% ----------------------------------------------------------------
\begin{abstract}
In this paper we study relation between some important ideals in the ring extension $R\subseteq T$, where $R$ is a certain maximal subring of a ring $T$. In fact, we would like to find some relation between $Nil_*(R)$ and $Nil_*(T)$, $Nil^*(R)$ and $Nil^*(T)$, $J(R)$ and $J(T)$, $Soc({}_RR)$ and $Soc({}_RT)$, and finally $Z({}_RR)$ and $Z({}_RT)$; especially, in certain cases, for example when $T$ is a reduced ring, $R$ (or $T$) is a left Artinian ring, or $R$ is a certain maximal subring of $T$. We show that either $Soc({}_RR)=Soc({}_RT)$ or $(R:T)_r$ (the greatest right ideal of $T$ which is contained in $R$) is a left primitive ideal of $R$. We prove that if $T$ is a reduced ring, then either $Z({}_RT)=0$ or $Z({}_RT)$ is a minimal ideal of $T$, $T=R\oplus Z({}_RT)$, and $(R:T)=(R:T)_l=(R:T)_r=ann_R(Z({}_RT))$. If $T=R\oplus I$, where $I$ is an ideal of $T$, then we completely determine the relation between Jacobson radicals, lower nilradicals, upper nilradicals, socle and singular ideals of $R$ and $T$. Finally, we study the relation between previous ideals of $R$ and $T$ whenever $R$ or $T$ is a left Artinian ring.
\end{abstract}
% ----------------------------------------------------------------
\section{Introduction}

\subsection{Motivation}
Let $R\subseteq T$ be an arbitrary ring extension, then there are many important results between certain ideals of $R$ and $T$, such as prime ideals, maximal ideals, primitive ideals, Jacobson radical and lower/upper nilradicals. Especially whenever $R\subseteq T$ is a commutative ring extension, integral extension of commutative rings, polynomial extension of a ring, matrix extension, finitely generated extension, central or normal extension, see \cite{good,lam,rvn,macrob}. In this paper we would like to study the properties between important ideals of $R$ and $T$, when $R$ is a certain maximal subring of a ring $T$, especially in certain cases that we mentioned in the abstract of the paper. If $T$ is a commutative ring and $R$ is a maximal subring of $T$, then it is clear that $N(R)=N(T)\cap R$, where $N(T)$ is the set of all nilpotent elements of a ring $T$. Also note that since in commutative case, either $R$ is integrally closed in $T$ or $T$ is integral over $R$ (i.e., $T$ is a finitely generated $R$-module), then we immediately infer that if $R$ is a maximal subring of $T$ and $T$ is integral over $R$, then $J(R)=J(T)\cap R$. In \cite[Section. 1 and Section. 2 of Chapter IV]{adjex}, the author proved some properties between certain ideals of $R$ and $T$ in a minimal ring extension $R\subseteq T$ of commutative rings (these extensions are called adjacent extension in \cite{adjex}). Also see \cite{akn}, for some other algebraic properties which are shared between a commutative ring and its maximal subrings. In \cite{dorsy} the authors studied minimal ring extensions of non-commutative rings. In fact, the authors generalized the results of \cite{dbsid} for non-commutative rings and characterized exactly (central) minimal ring extension of prime rings and therefore simple rings, see
\cite[Theorem 5.1, Corollary 5.3 and Theorem 6.1]{dorsy}. Also, they proved that a field $K$ has a minimal ring extension of the form $\mathbb{M}_n(D)$ where $D$ is a centrally finite division ring and $n>1$, if and only if $K$ has a proper subfield of finite index, see \cite[Lemma 6.6]{dorsy}. In \cite{azq} the author studied the existence of maximal subrings of non-commutative rings, and in \cite{azcond}, the conductor ideals of maximal subrings in non-commutative rings. We will mention some of needed results from \cite{azq} and \cite{azcond} in the preliminaries subsection. We refer the interested reader to \cite{mscnfsr}, for classification of maximal subrings of finite semisimple rings and application of them to calculating the covering number of a finite semisimple ring.

\subsection{Review of the results}
In Section 2, we consider the case that $R$ is a maximal subring of a ring $T$ such that $T=R\oplus I$, for an ideal $I$ of $T$. We prove that in this case $(R:T)_l=r.ann_R(I)$ and $(R:T)_r=l.ann_R(I)$. In particular, if $I^2\neq 0$ then $(R:T)=(R:T)_l=(R:T)_r$. We also determine when $T$ is a prime, semiprime and left primitive ring. A structure of maximal ideals of $T$ is given too. If $R$ is a prime ring which is a maximal subring of a ring $T$, then we determine $Min_T(0)$ and the conductor ideals of $R\subseteq T$. Moreover, if $R$ is a prime ring which is a maximal subring of a ring $T$ and $\mathcal{A}$ be the set of all nonzero ideals $I$ of $T$ such that $I\cap R=0$, then we prove that $|\mathcal{A}|\leq 2$ and $|\mathcal{A}|=2$ if and only if $T\cong\Delta(R)+(A\times A)$ (as subring of $R\times R$) for some minimal ideal $A$ of $R$ and in this case $Min_T(0)=\mathcal{A}$ and $(R:T)=(R:T)_l=(R:T)_r=0$. Conversely, we prove that if $R$ is a ring which is either prime or reduced, $A$ is a minimal ideal of $R$ and $T=\Delta(R)+(A\times A)$, then for an ideal $I$ of $T$ with $I\cap\Delta(R)=0$ and $I^2=I$, we have $I=A\times 0$ or $I=0\times A$.\\

In Section 3, we study the relation between socle, singular submodule, Jacobson radical, upper nilradical and lower nilradical in the ring extension $R\subseteq T$, where $R$ is a (certain) maximal subring of a ring $T$. We prove that if $T$ is a reduced ring and $R$ is a subring of $T$, then $Soc({}_RT)=Soc(T_R)$ is an ideal of $T$. In particular, if $R$ is a maximal subring of $T$, then either $Soc(R)=Soc({}_RT)$ or there exists a semisimple left (resp. right) $R$-submodule $A$ (resp. $B$) of $T$ such that $T=R\oplus A=R\oplus B$, moreover $(R:T)_l$ and $(R:T)_r$ are right primitive and left primitive ideals of $R$, respectively. Consequently, we prove that if there exists a simple left $R$-module $S$ of $T$, such that $S^2=S\nsubseteq R$, then there exists an idempotent $e$ of $T$ such that $T=R\oplus Re$. We show that if $R$ is a maximal subring of a reduced ring $T$, then $Z({}_RT)=Z(T_R)$ (which is an ideal of $T$, say $Z$) and either $Z=0$ or $T=R\oplus Z$, $(R:T)=(R:T)_l=(R:T)_r=ann_R(Z)$ and $Q=ann_T(Z)$ is a minimal prime ideal of $T$ and $Q\cap R=(R:T)$. We prove that if $R$ is a maximal subring of a prime ring $T$, then either $Z({}_RT)=0$ or $Soc({}_RR)=0$. If $R$ is a maximal subring of a ring $T$ and $T=R\oplus I$, where $I$ is an ideal of $T$, then we prove that either $J(R)=J(T)\cap R$, or $I^2=I$, $J(T)\subseteq (R:T)=(R:T)_l=(R:T)_r=l.ann_T(I)$, which is a primitive and minimal prime ideal of $T$, $dim(T)\geq 2$ and $I$ is unique. A similar result holds if we replace $Nil^*$ instead of Jacobson radical in the previous result. But fortunately, we always have $Nil_*(R)=Nil_*(T)\cap R$. We also prove that, either $Soc({}_RR)=Soc({}_RT)\subseteq (R:T)_r$ or $I$ is a semisimple left $R$-module which all of simple $R$-submodules of $I$ are isomorphic as $R$-module, $J(R)=J(T)\cap R$ and $Nil^*(R)=Nil^*(T)\cap R$. In particular, if either $R$ or $T$ is a one-sided Artinian ring, then $J(R)=J(T)\cap R$ and $Nil^*(R)=Nil^*(T)\cap R$. Moreover, if $R$ is a left Artinian ring, then $Soc({}_RI)=I$. We prove that if $R$ is a maximal subring of a left Noetherian ring $T$, then either $Nil^*(T)\subseteq Nil_*(R)$ or $Nil_*(R)=Nil^*(R)=Nil^*(T)\cap R$ is a nilpotent ideal of $R$. If in addition, $T$ is a left Artinian ring, then either $(R:T)$ is a maximal ideal of $T$, or $R$ is a zero-dimensional ring with only finitely many prime ideals.\\

Finally in Section 4, we focus on the left Artinian maximal subring $R$ of a ring $T$. First, we prove that $DCC$ holds on $Spec(T)$ and if $R$ does not contain any prime ideal of $T$, then $T$ is a zero-dimensional ring with only finitely many maximal ideals. Next, we observe that the algebraic properties of $T$ depend to the properties of $(R:T)_l$ in $T$ as follows: $(1)$ If $(R:T)_l$ is not an ideal of $T$ and $(R:T)_lT=T$, then we prove that ${}_RT$ and $((R:T)_l)_R$ are finitely generated. In particular, ${}_RT$ is a left Artinian $R$-module and hence a left Artinian ring. Moreover, either $J(R)=J(T)\cap R$ or $(R:T)\in Max(T)$, $T(R:T)_r=T$, $Max(R)=\{(R:T)_l, (R:T)_r\}\cup\{P\cap R\ |\ P\in Max(T), P\neq (R:T)\}$ and $|Max(R)|=|Max(T)|+1$. $(2)$ If $(R:T)_l$ is not an ideal of $T$ and $(R:T)_lT$ is a proper ideal of $T$, then we show that $(R:T)_lT\in Max(T)$, $|Max(T)|\leq |Max(R)|$, $R$ does not contains any maximal ideal of $T$, either $(R:T)_lT\in Min_T(0)$ and $J(R)=J(T)\cap R$ or $(R:T)_lT=r.ann_T(I)$ for some ideal $I$ of $T$ with $I^2=0$, and finally either $J(R)=J(T)\cap R$ or $(R:T)_r$ is not a finitely generated right ideal of $R$. $(3)$ $(R:T)_l$ is an ideal of $T$, which implies that $(R:T)_l=(R:T)_r=(R:T)$. In particular, we prove that in this case if $T$ is a zero-dimensional ring, then $J(R)=J(T)\cap R$. We show that $T/(R:T)\cong \mathbb{M}_n(S)$, where $S$ is a ring with at most two nonzero proper ideal (that are maximal ideals) and $S$ has a maximal subring which is a division ring. Moreover, either $T$ is a left Artinian ring or $S$ has a unique nonzero proper ideal.

\subsection{Notations and Definitions}
All rings in this paper are unital and all subrings, modules and homomorphisms are also unital. If $R\subsetneq T$ is a ring extension and there exists no other subring between $R$ and $T$, then $R$ is called a maximal subring of $T$, or the extension $R\subseteq T$ is called a minimal ring extension. If $T$ is a ring and $R$ and $S$ be subrings of $T$, then clearly $T$ is a $(R,S)$-bimodule and therefore we can consider $(R,S)$-subbimodules of $T$. In particular, if $t\in T$, then the $(R,S)$-subbimodule of $T$ which is generated by $t$ is denoted by $RtS=\{\sum_{i=1}^n r_its_i\ |\ r_i\in R, s_i\in S, n\geq 0\}$. It is clear that if $R\subseteq S$ and $t\in S$, then $R+RtS$ and $R+StR$ are subrings of $T$ which contain $R$. In particular, for each $t\in T$, the subrings $R+RtT$ and $R+TtR$ of $T$ contain $R$. Also note that, if $I$ is left (resp. right) ideal of $T$, then $R+IR$ (resp. $R+RI$) is a subring of $T$ which contains $R$. If $T$ is a ring, $I$ is a proper ideal of $T$ and $M$ is a left (resp. right) $T$-module, then $Min_T(I)$, $Max_r(T)$, $Max_l(T)$, $Max(T)$, $Spec(T)$, $l.ann_T(M)$ (resp. $r.ann_T(M)$) denote the set of all minimal prime ideals over $I$ in $T$, the set of all maximal right ideals of $T$, the set of all maximal left ideals of $T$, the set of all maximal ideals of $T$, the set of all prime ideals of $T$, the left annihilator of $M$ in $T$ (resp. the right annihilator of $M$ in $T$), respectively. The set of all left (resp. right) primitive ideals of $T$ is denoted by $Prm_l(T)$ (resp. $Prm_r(T)$). The characteristic of a ring $T$ is denoted by $Char(T)$. If $T$ is a ring, then $dim(T)$ denotes the classical Krull dimension of $T$ (i.e., the supremum of the lengths of all chains of prime ideals of $T$). If $T$ is a ring and $M$ is a left $T$-module, then $Soc({}_TM)$, $Z({}_TM)$ and $J({}_TM)$ denote the socle of $M$, the singular submodule of $M$ and the Jacobson radical of $M$, respectively (we have similar notations for the right $T$-modules). For a ring $T$, $Nil_*(T)$ and $Nil^*(T)$ are the lower nilradical and the upper nilradical of $T$, respectively. It is well known that $Nil_*(T)$ is the intersection of all prime ideals or only the minimal prime ideals of $T$. $Nil^*(T)$ is the unique maximal nil ideal of $T$ which is the sum of all nil ideals of $T$. A ring $T$ is called strongly prime if $T$ is a prime ring with no nonzero nil ideals. An ideal $P$ of a ring $T$ is called a strongly prime ideal if $T/P$ is a strongly prime ring. It is not hard to see that $Nil^*(T)$ is the intersection of all strongly prime ideals of $T$, see \cite[Proposition 2.6.7]{rvn}. If $A$ is a right ideal of a ring $T$, then $\mathbb{I}(A)$ denotes the idealizer of $A$ in $T$, that is the largest subring of $T$ which $A$ is a two-sided ideal in it, i.e., $\mathbb{I}(A)=\{t\in T\ |\ tA\subseteq A\}$, the similar notation is used when $A$ is a left ideal of $T$. For the properties of the previous notations and other used definitions and notations in this paper we refer the reader to \cite{lam} and \cite{rvn}.

\subsection{Preliminaries} In this subsection we review some results from \cite{azq} and \cite{azcond}. First we recall that if $R\subsetneq T$ is a ring extension which is finite as ring (module) over $R$, then one can easily see that $T$ has a maximal subring which contains $R$, by a natural use of Zorn's Lemma. Also, a proper subring $S$ of a ring $T$ is maximal if and only if for each $x\in T\setminus S$, we have $S[x]=T$. Moreover, if $R$ is a proper subring of $T$ and there exists $\alpha\in T$ such that $R[\alpha]=T$, then $T$ has a maximal subring $S$ with $R\subseteq S$ and $\alpha\notin S$. Now let $R$ be a ring and $\Delta(R):=\{(r,r)\ |\ r\in R\}$, be the diagonal subring of $T:=R\times R$. Then it is not hard too see that $\Delta(R)[(1,0)]=T$. Therefore $T$ has a maximal subring which contains $\Delta(R)$. In fact, subrings of $T$ which contains $\Delta(R)$ has the form $\Delta_I(R):=\Delta(R)+(I\times I)$, where $I$ is an ideal of $R$. In particular $\Delta_I(R)$ is a maximal subring of $T$ if and only if $I$ is a maximal ideal of $R$, see \cite[Theorem 2.7]{azq}. Therefore a ring $R$ is simple if and only if $\Delta(R)$ is a maximal subring of $R\times R$. More generally we have the following useful result.

\begin{thm}\label{pt1}
\cite[Theorem 2.9, Corollary 2.10]{azq}. Let $R_1$ and $R_2$ be rings and $T=R_1\times R_2$. Then $T$ has a maximal subring if and only if at least one of the following conditions holds:
\begin{enumerate}
\item $R_1$ or $R_2$ has a maximal subring.
\item There exist (maximal) ideals $I_j\subsetneq R_j$, for $j=1,2$, such that $\frac{R_1}{I_1}\cong \frac{R_2}{I_2}$ as rings.
\end{enumerate}
\end{thm}

Therefore if $R_1,\ldots, R_n$ are rings $n\geq 2$, then by the above theorem one can easily see that $T=R_1\times\cdots\times R_n$ has a maximal subring if and only if either there exists $i$ such that $R_i$ has a maximal subring or there exist $i\neq j$ and (maximal) ideals $I_k\subsetneq R_k$ for $k=i,j$ such that $\frac{R_i}{I_i}\cong \frac{R_j}{I_j}$ as rings. Moreover if each $R_i$ is simple and $S$ is a maximal subring of $T$, then there exists $j$ such that either $S\cong\prod_{i\neq j}R_i$ or $S\cong S_j\times \prod_{i\neq j}R_i$, where $S_j$ is a maximal subring of $R_j$.

\begin{thm}\label{pt2}
\cite[Theorem 2.13]{azq}. Let $S$ be a simple ring which is a maximal subring of a ring $R$, then every nonzero proper ideal of $R$ is maximal and $|Max(R)|\leq 2$. Moreover, exactly one of the following holds:
\begin{enumerate}
\item $R$ is simple ring which is a minimal ring extension of $S$.
\item $R$ has distinct maximal ideals $M$ and $N$ such that $S\cong \frac{R}{M}\cong\frac{R}{N}$ as rings, $M\cap N=0$ and $R\cong S\times S$ as rings.
\item $R$ has exactly one nonzero proper ideal $M$, $S\cong\frac{R}{M}$ as rings and $M^2=M$. In particular, either $R$ is a primitive ring or $J(R)=M$.
\item $R$ has exactly one nonzero proper ideal $M$, $S\cong\frac{R}{M}$ as rings and $M^2=0$. In particular, $J(R)=M$.
\end{enumerate}
\end{thm}

Let $R\subseteq T$ be a ring extension, then we define $(R:T):=\{x\in T\ |\ TxT\subseteq R\}$, $(R:T)_l:=\{x\in T\ |\ Tx\subseteq R\}$ and $(R:T)_r:=\{x\in T\ |\ xT\subseteq R\}$. In other words, $(R:T)$ is the largest ideal of $T$ which is contained in $R$, $(R:T)_l$ (resp. $(R:T)_r$)  is the largest left (resp. right) ideal of $T$ which is contained in $R$. It is not hard to see that $(R:T)_l=r.ann_R(T/R)$ and therefore $(R:T)_l$ is an ideal of $R$; similarly, $(R:T)_r=l.ann_R(T/R)$ is an ideal of $R$ (note, clearly $T/R$ is a left/right $R$-module). Hence $(R:T)_l$, $(R:T)_r$ and $(R:T)$ are ideals of $R$, which are called, left conductor ideal, right conductor ideal and the conductor ideal of the extension $R\subseteq T$ (or of $T$ in $R$), respectively. It is clear that $(R:T)_l(R:T)_r\subseteq (R:T)\subseteq (R:T)_l\cap (R:T)_r$. In particular, if $R\neq T$, then $(R:T)_l+(R:T)_r\subseteq R$ and therefore $(R:T)_l+(R:T)_r\subsetneq T$. Finally it is not hard to see that $(R:T)=r.ann_T(T/(R:T)_r)=l.ann_T(T/(R:T)_l)=l.ann_R(T/(R:T)_l)=r.ann_R(T/(R:T)_r)$, see \cite[(6) of Corollary 2.2]{azcond}. Moreover, we have the following result.

\begin{lem}\label{pt3}
\cite[Lemma 2.1]{azcond}. Let $R$ be a maximal subring of a ring $T$. Then $(R:T)_l$ and $(R:T)_r$ are prime ideals of $R$.
\end{lem}
\begin{proof}
Let $a,b\in R$ and $aRb\subseteq (R:T)_l$. Thus $TaRb\subseteq R$. Now assume that $a\notin (R:T)_l$, i.e., $Ta\nsubseteq R$. Thus $TaR\nsubseteq R$. Since $R$ is a maximal subring of $T$ we conclude that $R+TaR=T$ and thus $Tb=Rb+TaRb\subseteq R$, i.e., $b\in (R:T)_l$. Hence $(R:T)_l$ is a prime ideal of $R$. Similarly, $(R:T)_r$ is a prime ideal of $R$.
\end{proof}

We refer the reader to \cite{azcond}, for more results about the conductor ideals of maximal subrings (of a non-commutative rings).

\section{Prime ideals}
If $R$ is a (maximal) subring of a ring $T$, then it is clear that for each (nonzero proper) ideal $I$ of $T$ either $I\cap R=0$ or $I\cap R\neq 0$. In fact, if $R$ is a maximal subring of a ring $T$, then one of the following holds:
\begin{enumerate}
\item there exists a nonzero proper ideal $I$ of $T$ such that $R\cap I=0$. Therefore $R\oplus I=T$, for $R$ is maximal.
\item for each proper nonzero ideal $I$ of $R$, we have $R\cap I\neq 0$.
\end{enumerate}
In this section we determine when $T$ is semiprime, prime and left primitive ring in case $(1)$ of the above related to properties of $R$ and $I$. Next, we classify $T$ and $I$ when $R$ is a prime ring and also determine $Min_T(0)$. We remind the reader that the first part of item $(4)$ of the following theorem proved in \cite{dorsy}, we present it again for the sake of completeness and also for the presentation of the statement of it by the notation of conductor ideals.

\begin{thm}\label{t1}
Let $R$ be a maximal subring of a ring $T$. Assume that $T$ has a nonzero ideal $I$ such that $R\cap I=0$. Then the following hold:
\begin{enumerate}
\item $I$ is a minimal ideal of $T$. Moreover, $I$ is maximal respect to the property that $R\cap I=0$.
\item $(R:T)_l=r.ann_R(I)$ and $(R:T)_l=l.ann_R(I)$.
\item If $I^2\neq 0$ (in particular if $T$ is semiprime), then $(R:T)=(R:T)_l=(R:T)_r$.
\item  $T$ is a semiprime ring if and only if $R$ is semiprime and $I^2\neq 0$. In particular, in this case $Q:=ann_T(I)$ is a minimal prime ideal of $T$ and $Q\cap R=(R:T)=(R:T)_l=(R:T)_r$.
\item If $I^2=0$, then the contraction of prime (resp. semiprime, maximal, left primitive, right primitive, strongly prime) ideals of $T$ to $R$ remain prime (resp. semiprime, maximal, left primitive, right primitive, strongly prime). In particular, $J(R)\subseteq J(T)\cap R$, $Nil_*(R)\subseteq Nil_*(T)\cap R$ and $Nil^*(R)\subseteq Nil^*(T)\cap R$. Moreover $U(T)=U(R)+I$.
\item For each left (resp. right) ideal $A$ of $R$, $TA\cap R=A$ (resp. $AT\cap R=A$).
\item If $I^2=I$, then for each ideal $A$ of $R$ either $A\subseteq (R:T)$ or $IA=AI=I$ and $AT=TA$.
\end{enumerate}
\end{thm}
\begin{proof}
Since $R$ is a maximal subring of $T$, for each nonzero ideal $J\subseteq I$ of $T$, we infer that $T=R\oplus J$. This immediately shows that $I$ is a minimal ideal of $T$. Similarly, $I$ is maximal respect to property that $R\cap I=0$ and hence $(1)$ holds. As $T=R\oplus I$, we conclude that $T/R\cong I$ as left and right $R$-modules and therefore $(R:T)_l=r.ann_R(T/R)=r.ann_R(I)$ and $(R:T)_r=l.ann_R(T/R)=l.ann_R(I)$, thus $(2)$ holds. Now assume that $I^2\neq 0$. First note that $(R:T)_lI$ is an ideal of $T$ which is contained in $I$. Hence by minimality of $I$ we deduce that either $(R:T)_lI=I$ or $(R:T)_lI=0$. If $(R:T)_lI=I$, then by $(2)$ (i.e., $I(R:T)_l=0$), we have $I^2=((R:T)_lI)^2=0$ which is impossible. Thus $(R:T)_lI=0$ and hence $(R:T)_l\subseteq (R:T)_r$, by $(2)$. Similarly, $(R:T)_r\subseteq (R:T)_l$ and therefore $(3)$ is true. For $(4)$, first assume that $T$ is semiprime (hence $I^2\neq 0$). Let $N$ be an ideal of $R$ such that $N^2=0$. Thus by $(3)$ and $(2)$, $N\subseteq (R:T)=ann_R(I)$, for $(R:T)_l$ is a prime ideal of $R$. Thus $NI=IN=0$. Since $T=R\oplus I$, we immediately conclude that $N$ is an ideal of $T$, thus $N=0$, for $T$ is semiprime. Hence $R$ is semiprime. Conversely, assume that $R$ is semiprime and $I^2\neq 0$. We prove that $T$ is semiprime too. Suppose $A$ be a nonzero ideal of $T$ such that $A^2=0$. Thus $(A\cap R)^2=0$ and therefore $A\cap R=0$, for $R$ is semiprime. Hence $T=R\oplus A$, by maximality of $R$. Thus $R\oplus I=T=R\oplus A$. This immediately implies that $I\cong A$ as left and right $R$-modules. Hence $P:=l.ann_R(A)=l.ann_R(I)=r.ann_R(I)=r.ann_R(A)$, by $(3)$ and $(2)$. Since $A^2=0$ and $T=R\oplus A$, we can easily see that $l.ann_T(A)=P\oplus A=r.ann_T(A)$. Now note that since $I$ and $A$ are minimal ideals of $T$, we conclude that either $I\cap A= I$ or $I\cap A=0$. If $I\cap A=I$, then $I\subseteq A$ and thus $I^2=0$ which is absurd. Therefore $I\cap A=0$, hence $AI=IA=0$. Again by $T=R\oplus A$ and $AI=IA=0$, we deduce that $l.ann_T(I)=P\oplus A=r.ann_T(I)$. Hence $I\subseteq l.ann_T(A)=l.ann(I)$ and therefore $I^2=0$ which is a contradiction. Hence $T$ is a semiprime ring. Now assume that $T$ is a semiprime ring. If $I$ is contained in all minimal prime ideals of $T$, then since $T$ is a semiprime ring we infer that $I=0$, which is impossible. Hence there exists a minimal prime ideal $Q$ of $T$ such that $I\nsubseteq Q$. Since $I$ is a minimal ideal of $T$, we conclude that $I\cap Q=0$. Hence $IQ=QI=0$. From this one can easily see that $Q=ann_T(I)$. Therefore $Q\cap R=ann_R(I)$, which is equal to $(R:T)=(R:T)_l=(R:T)_r$, by $(3)$ and $(2)$. Thus $(4)$ holds. For $(5)$, assume that $I^2=0$, let $Q$ be a (semi)prime ideal of $T$, then clearly $I\subseteq Q$. Therefore $Q\nsubseteq R$ and hence $R+Q=T$, by maximality of $R$. Therefore $R/(R\cap Q)\cong T/Q$, as rings. This immediately implies that the contraction of (semi)prime ideals of $T$ to $R$ remain (semi)prime ideals in $R$. Since each maximal (resp. left primitive, right primitive, strongly prime) is a prime ideal of $T$, by a similar proof we conclude that for each maximal (resp. left primitive, right primitive, strongly prime) ideal $Q$ of $T$, the contraction $R\cap Q$ remains in $R$ with a same property as in $T$. From these one can easily see that the inclusions that mentioned in $(5)$ are valid. For instance we prove $Nil^*(R)\subseteq Nil^*(T)\cap R$. To see this, first note that $Nil^*(T)=\bigcap_{\alpha\in\Gamma}Q_{\alpha}$, where $\{Q_{\alpha}\ |\ \alpha\in \Gamma\}$ is the set of all strongly prime ideals of $T$. By the latter proof note that for each $\alpha\in\Gamma$, $Q_{\alpha}\cap R$ is a strongly prime ideal of $R$. Therefore $Nil^*(R)\subseteq Q_{\alpha}\cap R$, for each $\alpha\in\Gamma$. Hence $Nil^*(R)\subseteq Nil^*(T)\cap R$. For the final equality in $(5)$, note that $I\subseteq J(T)$, and $T=R\oplus I$ (which implies that $R\cong T/I$ as ring), therefore we deduce that an element $t=r+i$, where $r\in R$ and $i\in I$, is a unit in $T$ if and only if $r$ is a unit in $R$. Thus $U(T)=U(R)+I$ and hence we are done for $(5)$. For $(6)$, it is clear that $A\subseteq TA\cap R$. Hence assume that $x\in TA\cap R$. Thus there exist $t_i\in T$ and $a_i\in A$, $1\leq i\leq n$, such that $x=\sum_{i=1}^n t_ia_i$. Since $T=R\oplus I$, we conclude that each $t_i$ has the form $r_i+c_i$, where $r_i\in R$ and $c_i\in I$. Therefore $x-\sum_{i=1}^n r_ia_i=\sum_{i=1}^n c_ia_i\in R\cap I=0$. Thus $x=\sum_{i=1}^n r_i a_i\in A$ and therefore $A=TA\cap R$. For $(7)$, first note that by $(3)$, $(R:T)=(R:T)_l=(R:T)_r$. Hence $A\nsubseteq (R:T)$ if and only if $TA\nsubseteq R$ if and only if $AT\nsubseteq R$. Now assume that $A\nsubseteq (R:T)$. If $AI=0$, then from $T=R\oplus I$, we conclude that $AT=A\subseteq R$, which is absurd. Thus $AI\neq 0$ and similarly $IA\neq 0$. Since $AI$ and $IA$ are contained in $I$, $I\cap R=0$ and $R$ is a maximal subring of $T$, we deduce that $T=R\oplus AI=R\oplus IA$. Again, since $IA,AI\subseteq I$ and $T=R\oplus I$, we deduce that $IA=I=AI$ and therefore $TA=(R\oplus I)A=A\oplus IA=A\oplus AI=A(R\oplus I)=AT$.
\end{proof}

Note that in $(5)$ of the previous theorem, since $I$ is contained in the ideals $J(T)$, $Nil_*(T)$ and $Nil^*(T)$, then each of them are not contained in $R$. Therefore by maximality of $R$ we have $R+J(T)=T$, $R+Nil_*(T)=T$ and $R+Nil^*(T)=T$. Hence we have the natural ring isomorphisms $R/(R\cap J(T))\cong T/J(T)$, $R/(R\cap Nil_*(T))\cong T/Nil_*(T)$ and $R/(R\cap Nil^*(T))\cong T/Nil^*(T)$, similar to the proof of $(5)$ (for $Q$) in the previous theorem. In the next result we determine when $T$ is a prime ring in case $(1)$ that mentioned in the beginning of this section.

\begin{prop}\label{t2}
Let $R$ be a maximal subring of a ring $T$. Assume that $T$ has a nonzero ideal $I$ such that $R\cap I=0$. Then the following are equivalent:
\begin{enumerate}
 \item $R$ is a prime ring and $l.ann_T(I)=0$.
 \item $R$ is a prime ring and $r.ann_T(I)=0$.
 \item $T$ is a prime ring and $I\in Spec(T)$.
\end{enumerate}
Moreover, if any of these equivalent conditions holds, then $(R:T)=(R:T)_l=(R:T)_r=0$.
\end{prop}
\begin{proof}
We prove $(1)\Longleftrightarrow (3)$. The proof of $(2)\Longleftrightarrow (3)$ is similar. Since $R$ is a maximal subring of $T$ and $I$ is a nonzero ideal of $T$ which is not contained in $R$ we have $R\oplus I=T$ and therefore $R\cong T/I$ as rings, for $R\cap I=0$. This immediately implies that $(3)\Longrightarrow (1)$. Hence assume that $(1)$ holds. It is obvious that $I$ is a prime ideal of $T$ for $T/I\cong R$ (as rings) and $R$ is a prime ring. Now assume that $A$ and $B$ are ideals of $T$ such that $AB=0$. Since $I$ is a prime ideal of $T$, the we conclude that either $A\subseteq I$ or $B\subseteq I$. If $A\subseteq I$ and $A\neq 0$, then by minimality of $I$ (by $(1)$ of the previous theorem), we deduce that $A=I$ and therefore $IB=0$. From $l.ann_T(I)=0$, we conclude that $I^2\neq 0$ and therefore $T$ is a semiprime ring by $(4)$ of the previous theorem. Since $IB=0$, we conclude that $(BI)^2=0$ and therefore $BI=0$, for $T$ is a semiprime ring. Thus $B\subseteq l.ann_T(I)=0$. Similarly, if $B\subseteq I$ and $B\neq 0$, we conclude that $A=0$. Hence in any cases we deduce that $A=0$ or $B=0$, i.e., $T$ is prime. The equalities $(R:T)=(R:T)_l=(R:T)_r=0$ follows from $(2)$ and $(3)$ of the previous theorem.
\end{proof}

Next we would like to determine when $T$ is a left/right primitive ring (for case $(1)$ that mentioned in the beginning of this section). First we need the following two lemmas.

\begin{lem}\label{t3}
Let $T$ be a ring and $I$ be a left ideal of $T$. Then either $I$ has a maximal left $T$-submodule or $I\subseteq J(T)$. In particular, if $J(T)=0$, then each nonzero left ideal of $T$ has a maximal left $T$-submodule.
\end{lem}
\begin{proof}
Assume that $I$ is not contained in $J(T)$. Thus there exists a left maximal ideal $M$ of $T$ such that $I\nsubseteq M$. Therefore by maximality of $M$ we conclude that $M+I=T$. Hence $I/(I\cap M)\cong T/M$ as left $T$-modules. Thus $I\cap M$ is a maximal submodule of $I$. The final part is evident.
\end{proof}

\begin{lem}\label{t4}
Let $T$ be a ring and $I$ be a nonzero minimal proper ideal of $T$. Then the following hold:
\begin{enumerate}
\item if $I$ has a maximal left $T$-submodule, then $l.ann_T(I)$ is a left primitive ideal of $T$.
\item if $T$ is a left Artinian ring, then $l.ann_T(I)\neq 0$ is a maximal ideal of $T$.
\item if $T$ is a left Noetherian ring, then either $T$ is a left primitive ring, or $l.ann_T(I)\neq 0$ is a left primitive ideal.
\end{enumerate}
\end{lem}
\begin{proof}
$(1)$ Let $N$ be a maximal left $T$-submodule of $I$. Then clearly $P=l.ann_T(I/N)$ is a left primitive ideal of $T$. If $P=l.ann_T(I)$,  then we are done. Hence assume that $PI\neq 0$. By minimality of $I$, we infer that $PI=I$. But clearly $PI\subseteq N$. Hence $I\subseteq N$ which is absurd. Thus $P=l.ann_T(I)$ is a left primitive ideal of $T$. $(2)$ and $(3)$ are immediate consequences of $(1)$, for each cases $I$ has a maximal $T$-submodule, and therefore $P=l.ann_T(I)$ is a left primitive ideal of $T$. For $(3)$ note that, if $P=0$, then $T$ is a left primitive ring and otherwise $P=l.ann_T(I)\neq 0$ is a left primitive ideal of $T$.
\end{proof}

Now the following is in order.

\begin{prop}\label{t5}
Let $R$ be a maximal subring of a ring $T$. Assume that $T$ has a nonzero ideal $I$ such that $R\cap I=0$. Then the following are equivalent:
\begin{enumerate}
\item $R$ is a left primitive ring, $l.ann_T(I)=0$ and $I$ has a maximal left $T$-submodule.
\item $T$ is a left primitive ring and $I$ is a left primitive ideal of $T$.
\end{enumerate}
Moreover, if any of these equivalent conditions holds, then $(R:T)=(R:T)_l=(R:T)_r=0$.
\end{prop}
\begin{proof}
First note that since $R$ is a maximal subring of $T$ and $T=R\oplus I$, then we conclude that $T/I\cong R$ as rings. $(1)\Longrightarrow (2)$ Since $R$ is a left primitive we immediately conclude that $I$ is a left primitive ideal of $T$. Also note that by $(1)$ of the previous lemma, $T$ is a left primitive ring. Conversely, assume that $(2)$ holds. Clearly $R$ is a left primitive ring and $l.ann_T(I)=0$ for $T$ is prime. If $I$ has no maximal $T$-submodules, then we deduce that $I\subseteq J(T)=0$, by Lemma \ref{t3}, which is absurd. The equalities $(R:T)=(R:T)_l=(R:T)_r=0$ follows from $(2)$ and $(3)$ of Theorem \ref{t1}.
\end{proof}

Let $S$ be a ring and $\mathcal{P}$ be an algebraic property which preserved by ring isomorphisms. Then we define the set $\Omega_{\mathcal{P}}(S)$ be the set of all (proper) ideals of $S$ such that the ring $S/I$ has property $\mathcal{P}$. For example if $\mathcal{P}$ is the property that a ring is prime, then $\Omega_{\mathcal{P}}(S)=Spec(S)$. Now the following is in order.

\begin{rem}\label{t6}
Let $R$ be a maximal subring of a ring $T$ and $I$ be a nonzero proper ideal of $T$ such that $R\cap I=0$. Therefore $T=R\oplus I$. If $\mathcal{P}$ is an algebraic property which preserved by ring isomorphisms, then $\{K\oplus I\ |\ K\in\Omega_{\mathcal{P}}(R)\}\subseteq \Omega_{\mathcal{P}}(T)$. To see this, note that one can easily see that $K\oplus I$ is an ideal of $T$ and $T/(K\oplus I)=(R\oplus I)/(K\oplus I)\cong R/K$ as rings, and therefore since the property $\mathcal{P}$ is preserved by ring isomorphisms, we immediately conclude the inclusion. In particular, the following hold:
\begin{enumerate}
\item $\{P\oplus I\ |\ P\in Spec(R)\}\subseteq Spec(T)$.
\item $\{M\oplus I\ |\ M\in Max(R)\}\subseteq Max(T)$.
\item $\{P\oplus I\ |\ P\in Prm_l(R)\}\subseteq Prm_l(T)$.
\end{enumerate}
The similar inclusions hold for right primitive ideals, primitive ideals and strongly prime ideals. Moreover, $\{M\oplus I\ |\ M\in Max_l(R)\}\subseteq Max_l(T)$ (the same inclusion holds for $Max_r$ instead of $Max_l$). To see this note that one can easily see that $M\oplus I$ is a proper left ideal of $T$ and $T/(K\oplus I)=(R\oplus I)/(K\oplus I)\cong R/K$ holds as left $R$-modules. This immediately shows that $K\oplus I$ is a maximal left ideal of $T$.
\end{rem}

\begin{prop}\label{t7}
Let $R$ be a maximal subring of a ring $T$ and $I$ be a nonzero proper ideal of $T$ such that $R\cap I=0$. If $M$ is a maximal ideal of $T$, then either $M=(M\cap R)\oplus I$ or $I^2=I$ is not contained in $M$ and one of the following holds:
\begin{enumerate}
\item $(R:T)=(R:T)_l=(R:T)_r=M=ann_T(I)$ is a minimal prime ideal of $T$.
\item $(R:T)=R\cap M$ is a maximal ideal of $R$ and $(R:T)\oplus I$ is a maximal ideal of $T$.
\end{enumerate}
\end{prop}
\begin{proof}
If $M$ contains $I$, then $M$ is not contained in $R$ and therefore by maximality of $R$ we deduce that $R+M=T$. Therefore $R/(R\cap M)\cong T/M$ as rings. Thus $R\cap M$ is a maximal ideal of $R$, therefore by the previous remark we deduce that $(R\cap M)\oplus I$ is a maximal ideal of $T$. Since $(R\cap M)\oplus I\subseteq M$, we infer that the equality holds. Hence assume that $M$ does not contain $I$. By $(1)$ of Theorem \ref{t1}, $I$ is a minimal ideal of $T$, hence either $I^2=0$ or $I^2=I$. The case $I^2=0$, implies that $I\subseteq M$, which is impossible. Thus $I^2=I$ and therefore by $(3)$ of Theorem \ref{t1}, we conclude that $(R:T)=(R:T)_l=(R:T)_r$. Again by minimality of $I$ and the fact that $M$ does not contains $I$ we deduce that $M\cap I=0$ and therefore $IM=MI=0$. Hence $M=l.ann_T(I)=r.ann_T(I)=ann_T(I)$. If $Q$ is a prime ideal of $T$, which properly is contained in $M$, then $MI=0\subseteq Q$, and therefore $I\subseteq Q\subsetneq M$, which is absurd. Thus $M$ is a minimal prime ideal of $T$. Now we have two cases:
\begin{enumerate}
\item $M\subseteq R$ and therefore $M=(R:T)$.
\item $M$ is not contained in $R$. Therefore by maximality of $R$ we deduce that $R+M=T$ and hence $R/(R\cap M)\cong T/M$ as rings. Thus $R\cap M$ is a maximal ideal of $R$. Now since $R\oplus I=T$ and $MI=0$, we infer that $T(R\cap M)=R\cap M$, and hence $R\cap M\subseteq (R:T)$. Thus $R\cap M=(R:T)$ is a maximal ideal of $R$. Clearly $(R:T)\oplus I$ is a maximal ideal of $T$, by the previous remark.
\end{enumerate}
\end{proof}

In the next result we want to exactly determine the set of minimal prime ideals of a ring $T=R\oplus I$, where $I$ is an ideal of $T$ and $R$ is a prime maximal subring of $T$.

\begin{thm}\label{t14}
Let $R$ be a prime ring which is a maximal subring of a ring $T$. Then either for each nonzero ideal $I$ of $T$, $R\cap I\neq 0$, or there exists a nonzero (prime) ideal $I$ of $T$ such that $R\cap I=0$ and one of the following holds:
\begin{enumerate}
\item $T$ is a prime ring. In particular, $(R:T)=(R:T)_l=(R:T)_r=0$.
\item $I^2=0$ and $Min_T(0)=\{I\}$.
\item $I^2=I$, $T$ is semiprime, $Min_T(0)=\{I, ann_T(I)\}$. Moreover, $(R:T)=(R:T)_l=(R:T)_r=ann_T(I)\cap R$.
\end{enumerate}
\end{thm}
\begin{proof}
Assume that there exists a nonzero ideal $I$ of $T$ such that $R\cap I=0$. Clearly, $I$ must a prime ideal of $T$, for $R\cong T/I$ as rings. If $T$ is prime, then $ann_T(I)=0$ and therefore by $(2)$ of Theorem \ref{t1}, $(1)$ holds. Hence assume that $T$ is not prime. By $(1)$ of Theorem \ref{t1}, $I$ is a minimal ideal of $T$ and hence either $I^2=0$ or $I^2=I$. If $I^2=0$, then clearly $Min_T(0)=\{I\}$ and therefore $(2)$ holds. Thus assume that $I^2=I$. By $(4)$ of Theorem \ref{t1}, $T$ is a semiprime ring and $Q=ann_T(I)$ is a minimal prime ideal of $T$. Thus by Proposition \ref{t2} we conclude that $Q\neq 0$, for $T$ is not prime. Now note that since $I^2=I\neq 0$, we deduce that $I\nsubseteq Q$. Therefore by minimality of $I$ we conclude that $I$ and $Q$ are incomparable. Again by minimality of $I$ we obtain that $I$ is a minimal prime ideal of $T$. From $Iann_T(I)=0$, we deduce that $Min(T)=\{I,Q\}$. The final part of $(3)$ is evident by $(4)$ of Theorem \ref{t1}.
\end{proof}

In the next result we exactly determine when in the case $(3)$ of the previous result we have $(R:T)=(R:T)_l=(R:T)_r=0$.

\begin{prop}\label{t15}
Let $R$ be a prime ring which is a maximal subring of a ring $T$. Let $\mathcal{A}$ be the set of nonzero ideals $I$ of $T$, such that $R\cap I=0$. Then $|\mathcal{A}|\leq 2$. In fact, $|\mathcal{A}|=2$ if and only if there exists a minimal ideal $A$ of $R$ such that $T\cong \Delta(R)+(A\times A)$ as subring of $R\times R$. Moreover, in this case $Min(T)=\mathcal{A}$ and $(R:T)=(R:T)_l=(R:T)_r=0$.
\end{prop}
\begin{proof}
Since for each $I\in \mathcal{A}$, $T/I\cong R$ as rings, and $R$ is a prime ring, we infer that each element of $\mathcal{A}$ is a prime ideal of $T$ which is a minimal ideal of $T$, by $(1)$ of Theorem \ref{t1}. Hence, if $I\neq J$ are in $\mathcal{A}$, then $I\cap J=0$ and therefore $IJ=0$. This immediately shows that $\mathcal{A}=\{I, J\}$ and also $I=ann_T(J)$ and $J=ann_T(I)$, by $(4)$ of Theorem \ref{t1} (or $(3)$ of the previous theorem). Therefore by the previous theorem we conclude that $Min_T(0)=\mathcal{A}$ and $(R:T)=(R:T)_l=(R:T)_r= ann_T(I)\cap R= J\cap R=0$. Finally note that since $I\cap J=0$, we conclude that there exists a natural ring embedding, say $\phi$, such that $T\hookrightarrow T/I\times T/J=R\times R$ and since $R$ is a maximal subring of $T$, by the comments in Subsection $1.4$ of Preliminaries in the introduction of this paper, we conclude that there exists a minimal ideal $I$ of $R$ such that $\phi(T)=\Delta(R)+(A\times A)$, hence we are done.
\end{proof}

It is clear that in the previous result, if $I\in\mathcal{A}$ and $I^2=0$, then $\mathcal{A}=\{I\}$. In other words, $I$ is unique ideal of $T$ respect to the property that $R\cap I=0$. Also note that if $R$ is not prime then either $|\mathcal{A}|\leq 1$ or $T\cong \Delta(R)+(A\times A)$, for some minimal ideal $A$ of $R$. Now the following is in order.

\begin{prop}\label{t16}
Let $R$ be a ring, $A$ be a minimal ideal of $R$ and $T=\Delta(R)+(A\times A)$, as subring of $R\times R$. Assume that $I$ be an ideal of $T$ such that $\Delta(R)\cap I=0$ and $I^2=I$. If at least one of the following conditions holds, then either $I=A\times 0$ or $I=0\times A$ (and therefore in any cases, $A^2=A$).
\begin{enumerate}
\item $R$ is a prime ring.
\item $R$ is a reduced ring or $I$ is not nil.
\item $2\notin ann(I)$. In particular in this case,  $I$ is non-commutative (i.e., there exist $a,b\in I$ such that $ab\ne ba$).
\item $Char(R)=0$ and $(\Delta(R): T)=0$ (or only, $2T\nsubseteq \Delta(R)$).
\item $Char(R)=n>0$ is odd.
\end{enumerate}
\end{prop}
\begin{proof}
First note that since $A$ is a minimal ideal of $R$, then by the comments in Subsection $1.4$ of the preliminaries in the introduction of this paper, $\Delta(R)$ is a maximal subring of $T=\Delta(R)+(A\times A)$. Therefore by maximality of $\Delta(R)$ we deduce that $\Delta(R)\oplus I=T$, $\Delta(R)\oplus (A\times 0)=T$ and $\Delta(R)\oplus (0\times A)=T$. Hence by $(1)$ of Theorem \ref{t1}, we conclude that $I$, $A\times 0$, and $0\times A$ are minimal ideals of $T$. If $I\cap (A\times 0)\neq 0$, then by minimality of $I$ we obtain that $I=A\times 0$. Thus we may assume that $I\cap(A\times 0)=0$ and similarly $I\cap(0\times A)=0$. Hence by notation of Proposition \ref{t15}, we have $|\mathcal{A}|\geq 3$ and therefore we conclude that $R$ is not a prime ring. For the next part, let $(x,0)\in I$, then $(x,0)\in T=\Delta(R)\oplus (A\times 0)$. Thus there exist $r\in R$ and $a\in A$ such that $(x,0)=(r,r)+(a,0)$. Therefore $r=0$ and $x=a\in A$. Hence $(x,0)\in I\cap(A\times 0)=0$, i.e., $x=0$. Similarly if $(0,y)\in I$, then $y=0$. Now let $(x,y)\in I$, since $(x,x), (y,y)\in\Delta(R)\subseteq T$ and $I$ is an ideal of $T$, we conclude that $(x^2,xy),(x^2,yx),(xy,y^2), (yx,y^2)\in I$. Clearly $(x^2,y^2)\in I$ and therefore $(0,xy-y^2)\in I$. Hence $xy=y^2$ and therefore $(xy,y^2)\in \Delta(R)\cap I=0$. Thus $y^2=xy=0$. Similarly $x^2=xy=0$, $x^2=yx=0$ and $y^2=yx=0$. Hence $x^2=y^2=xy=yx=0$ and $(x,y)^2=0$, i.e., $I$ is nil. Therefore if $I$ is not nil or if $R$ is a reduced ring (note that if $R$ is a reduced ring, then $R\times R$ and hence $T$ is a reduced ring) we are done. Thus if $(2)$ holds we are done. For $(3)$, note that as we see in the previous part, for each $a\in I$, we have $a^2=0$. This immediately implies that for each $a,b\in I$ we have $ab=-ba$ (note, $a+b\in I$ and therefore $(a+b)^2=0$). Thus for each $a,b, c\in I$, we have $abc=-bac=bca=-abc$ and therefore $2abc=0$. Hence $2I^3=0$. Now since $I=I^2$, we conclude that $I=I^3$ and therefore $2I=0$. Thus if $2\notin ann(I)$, we have a contradiction and hence we are done for the first part of $(3)$. for the final part of $(3)$, note that if $I$ is commutative, then for each $a, b\in I$, we have $ba=ab=-ba$ and therefore $2ba=0$, i.e., $2I=0$ which is absurd. Finally for $(4)$ and $(5)$, since $T=\Delta(R)\oplus I$ and $2I=0$, we conclude that $2T=2\Delta(R)\subseteq \Delta(R)$. Thus, if $Char(R)=0$, then $0\neq 2T\subseteq (\Delta(R):T)$, which is a contradiction by $(4)$. Thus suppose that $Char(T)=n>0$ is odd, then $2T=T\subseteq \Delta(R)$, which is impossible by $(5)$. Thus if each of cases $(1)-(5)$ holds, we deduce that either $I=A\times 0$ or $I=0\times A$.
\end{proof}

\section{Socle, singularity and radicals}
Let $R$ be a (certain) maximal subring of a ring $T$. In this section first we study the relation between the socles of ${}_RR$, ${}_RT$ and $T_R$. Similarly, for singular submodules of ${}_RR$ and ${}RT$. Next we investigate about the relation between radical ideals (i.e., Jacobson radical, upper nilradical and lower nilradical) of $R$ and $T$, when $R$ is a maximal subring of $T$, especially, when $T$ is of the form $T=R\oplus I$, where $I$ is an ideal of $T$. It is well known that if $T$ is a semiprime ring, in particular if $T$ is a reduced ring, then $Soc({}_TT)=Soc(T_T)$. In the following result we have a more result. In fact we show that if $T$ is a reduced ring and $R$ be any subring of $T$, then $Soc({}_RT)=Soc(T_R)$ is an ideal of $T$.

\begin{lem}\label{t8}
Let $R$ be a subring of a reduced ring $T$. Then the following hold:
\begin{enumerate}
\item for each $x\in T$, $r.ann_R(x)=l.ann_R(x)$. In particular, $Rx$ is a simple left $R$-module if and only if $xR$ is a simple right $R$-module.
\item for each $x\in T$ and $n\geq 1$, $l.ann_R(x)=l.ann_R(x^n)$. In particular, $Rx$ is simple if and only if $Rx^n$ is simple.
\item $Soc({}_R T)=Soc(T_R)$ is an ideal of $T$.
\item if $I$ is a simple left $R$-submodule of $T$, then $I^2=I$ if and only if there exists a central idempotent $e$ of $T$ such that $I=Re$.
\end{enumerate}
\end{lem}
\begin{proof}
First note that since $T$ is a reduced ring, we conclude that whenever $a,b\in T$, then $ab=0$ if and only if $ba=0$. To see this, assume that $ab=0$, then $(ba)^2=0$ and therefore $ba=0$. This immediately implies that $A:=r.ann_R(x)=l.ann_R(x)$ and therefore $A$ is an ideal of $R$. Hence $Rx$ is a simple left $R$-module if and only if $xR$ is a simple right $R$-module. Hence $(1)$ holds. To see $(2)$, it suffices to show that $l.ann_R(x^n)\subseteq l.ann_R(x)$. We prove by induction on $n\geq 2$. For $n=2$, if $rx^2=0$, then $(rx)x=0$. Thus $x(rx)=0$ and hence $(rx)^2=rxrx=0$. Therefore $rx=0$, for $T$ is a reduced ring. Hence assume that it is true for $n-1$, we prove it for $n$. Suppose that $rx^n=0$. Therefore $(rx)x^{n-1}=0$. Hence $rx\in l.ann_R(x^{n-1})=l.ann_R(x)$. Therefore $rx^2=0$ and thus $rx=0$. Hence $(2)$ holds. For $(3)$ it suffices to show that $Soc({}_R T)$ is a right ideal of $T$. To see this, let $S$ be a simple left $R$-submodule of $T$ and $t\in T$. Then $St$ (which is a homomorphic image of $S$) is either $0$ or is isomorphic to $S$ as left $R$-module. Thus $St\subseteq Soc({}_R T)$. Hence $Soc({}_R T)$ is a right ideal of $T$. Similarly $Soc(T_R)$ is a left ideal of $T$. Now by $(1)$ we immediately conclude that $Soc({}_R T)=Soc(T_R)$ and therefore $(3)$ holds. Finally, for $(4)$ it is clear that if $e$ is a central idempotent of $T$ such that $I=Re$, then $I^2=I$. Conversely, assume that $I$ is a simple left $R$-submodule of $T$, such that $I^2=I$. Assume that $x\in I$ and $x\neq 0$. Then clearly, $0\neq x^2\in Ix$ and therefore $0\neq Ix\subseteq I^2\subseteq I$. Hence $Ix=I$, for $I$ is simple. Therefore there exists $a\in I$ such that $x=ax$. Thus $ax=a^2x$. Now note that $a-a^2\in I$, for $I^2=I$ and $a-a^2\in l.ann_T(x)$. Thus $a-a^2\in I\cap l.ann_T(x)$. Since $I$ is simple and $Ix=I$, we immediately deduce that $I\cap l.ann_T(x)=0$ and therefore $a^2=a$. Clearly $I=Ra$ and $a$ is central in $T$ (for $T$ is reduced).
\end{proof}

By $(3)$ of the previous lemma we have the following immediate corollary.

\begin{cor}
Let $R$ be a subring of a reduced ring $T$. Then ${}_RT$ is semisimple if and only if $T_R$ is semisimple.
\end{cor}

By assuming $R$ is a maximal subring of a ring $T$ in Lemma \ref{t8} we have the following result.

\begin{prop}\label{t9}
Let $R$ be a maximal subring of a ring $T$. Then either $Soc({}_RR)=Soc({}_RT)$ (resp. $Soc(R_R)=Soc(T_R)$) or there exists a semisimple left (resp. right) $R$-submodule $A$ (resp. $B$) of $T$ such that $T=R\oplus A$ (resp. $T=R\oplus B$), in particular in the latter case, $(R:T)_r$ (resp. $(R:T)_l$) is a left (resp. right) primitive ideal of $R$.
\end{prop}
\begin{proof}
Assume that $Soc({}_RR)\neq Soc({}_RT)$. Hence there exists a semisimple left $R$-submodule $A$ of $T$ such that $Soc({}_R T)=Soc({}_R R)\oplus A$. Since $R$ is a maximal subring of $T$ and $Soc({}_R T)$ (by a similar proof of the previous theorem) is a right $T$-submodule of $T$ (i.e., a right ideal of $T$) which also a left $R$-submodule of $T$ and is not contained in $R$, we deduce that $R+Soc({}_R T)=T$. Hence $R+(Soc({}_R R)\oplus A)=T$. Also, note that $R\cap A=0$, for $R\cap A$ is a semisimple left $R$-submodule which is contained in $Soc({}_RR)\cap A=0$. Therefore $R\oplus A=T$. Finally, note that $R+Soc({}_RT)=T$, implies that $T/R\cong Soc({}_R T)/Soc({}_RR)$ as left (and right) $R$-module. Hence by \cite[Corollary 2.5]{azcond}, we immediately deduce that $(R:T)_r$ is a left primitive ideal of $R$, for $Soc({}_R T)/Soc({}_RR)$ is a nonzero semisimple left $R$-module.
\end{proof}

Note that by $(3)$ of Lemma \ref{t8}, if $R$ is a subring of a reduced ring $T$, then $Soc({}_R T)=Soc(T_R)$ is an ideal of $T$, we denote it by $Soc_R(T)$ (in particular, if $R=T$, then we denote it by $Soc(R)$). Now the following is in order.

\begin{cor}\label{t10}
Let $R$ be a maximal subring of a reduced ring $T$. If $Soc(R)\neq Soc_R(T)$, then there exists $x\in T\setminus R$, such that $Rx$ and $xR$ are simple left and right $R$-submodules of $T$, respectively and $T=R+RxT=R+TxR$. In particular, $RxT$ and $TxR$ are semisimple left and right $R$-submodules of $T$, respectively. Moreover, if $I^2=I\nsubseteq R$ is a simple left (right) $R$-submodule of $T$, then $T=R\oplus Re$ and $Soc_R(T)=Soc(R)\oplus Re$.
\end{cor}
\begin{proof}
Let $Rx$ be a simple left $R$-submodule of $T$ which is not contained in $R$. Clearly, $Rx\cap R=0$ and by $(1)$ of Lemma \ref{t8}, $xR$ is a simple right $R$-module of $T$ and therefore $xR\cap R=0$. Hence $RxT$ and $TxR$ are contained in $Soc_R(T)$ but not in $R$ (see the proof of $(3)$ of Lemma \ref{t8}). Hence $RxT$ and $TxR$ are semisimple left and right $R$-submodules of $T$, respectively. Thus by the maximality of $R$, we have $R+RxT=T=R+TxR$. Now let $I$ be a simple left $R$-module of $T$, $I\nsubseteq R$ and $I^2=I$. By $(4)$ of Lemma \ref{t8}, there exists a central idempotent $e$ of $T$, such that $I=Re$ (hence $I=eR$ is a simple right $R$-submodule of $T$). Clearly $R\cap I=0$, for $I$ is a simple which is not contained in $R$. Thus $R+I=R\oplus Re$ is a subring of $T$ (note that $I^2=I$) which properly contains $R$. Thus we conclude that $T=R\oplus Re$. Finally note that $Soc_R(T)=(Soc(R)\oplus Re)\oplus A$ for some right $R$-submodule $A$ of $T$ and therefore $R\oplus Re=T=R+Soc_R(T)=R\oplus (Re\oplus A)$. Hence if $x\in A$, then $x=r+ae$ for some $r,a\in R$ and therefore $x-ae=r\in R\cap (Re\oplus A)=0$. Hence $x=ae\in Re$ and therefore $A=0$, i.e., $Soc_R(T)=Soc(R)\oplus Re$.
\end{proof}

Next we want to study the relation between $Z({}_RT)$ and $Z({}_RR)$ for a (maximal) subring $R$ of a ring $T$. First we need the following lemma.

\begin{lem}\label{t11}
Let $R$ be a subring of a ring $T$. Then $Z({}_R T)$ is a proper $(R,T)$-subbimodule of $T$ which contains $Z({}_R R)$. In particular, if $R$ is a maximal subring of $T$, then either $Z({}_RT)=Z({}_RR)$ or $R+Z({}_RT)=T$ and therefore $R/Z({}_RR)\cong T/Z({}_RT)$ as left and right $R$-modules; moreover, in this case $TZ({}_RR)$ and $Z({}_RR)T$ are contained in $Z({}_R T)$. In particular, if $T$ is a prime ring, $R$ is a maximal subring of $T$ and $Z({}_RT)\neq Z({}_RR)$, then either $Soc({}_RR)=0$ or $Z({}_RR)=0$.
\end{lem}
\begin{proof}
It is clear that $Z:=Z({}_RT)$ is a left $R$-submodule of $T$. Also note that for each $a\in Z$ and $t\in T$, $l.ann_R(a)\subseteq l.ann_R(at)$ and therefore $at\in Z$. Thus $Z$ is a $(R,T)$-subbimodule of $T$. Clearly $1\notin Z$ and therefore $Z$ is proper. It is obvious that $Z({}_RR)\subseteq Z$ and $Z\cap R=Z({}_RR)$. Now assume that $R$ is a maximal subring of $T$ and $Z({}_RR)\neq Z$. Thus by maximality of $R$ we conclude that $R+Z=T$. Since in the latter equality all of the parts are left and right $R$-submodules of $T$, we obtain that $R/Z({}_RR)\cong T/Z$ as left and right $R$-modules. Now since $Z({}_RR)$ is an ideal of $R$, we deduce that $l.ann_R(R/Z({}_RR))=r.ann_R(R/Z({}_RR))=Z({}_RR)$. Hence $l.ann_R(T/Z)=r.ann_R(T/Z)=Z({}_RR)$. This immediately shows that $TZ({}_RR)$ and $Z({}_RR)T$ are contained in $Z$. Since $Soc({}_RR)Z=0$, we infer that $Soc({}_RR)TZ({}_RR)=0$. The final part is evident now.
\end{proof}

Similarly $Z(T_R)$ is a proper $(T,R)$-subbimodule of $T$, which contains $Z(R_R)$. We remind the reader that a ring $T$ is called symmetric if for each natural number $n$ and $a_1,\ldots,a_n\in T$, whenever $a_1a_2\cdots a_n=0$, then for each permutation $\sigma\in S_n$, we have $a_{\sigma(1)}a_{\sigma(2)}\cdots a_{\sigma(n)}=0$. Note that by \cite[Theorem I.3]{andcam}, each reduced ring $T$ is a symmetric ring. Now we have the following main result.

\begin{thm}\label{t12}
Let $R$ be a maximal subring of a reduced ring $T$. Then $Z:=Z({}_R T)=Z(T_R)$ is a proper ideal of $T$. Moreover, either $Z=0$ or the following hold:
\begin{enumerate}
\item $R\oplus Z=T$. In particular, $Z$ is a minimal ideal of $T$. In fact for each $0\neq x\in Z$, we have $RxT=Z=TxR$.
\item $(R:T)=(R:T)_l=(R:T)_r=ann_R(Z)$ which contains $Soc(R)$. Moreover, $Q=ann_T(Z)$ is a minimal prime ideal of $T$ and $Q\cap R=(R:T)$.
\item for each minimal prime ideal $P$ over $Z$ in $T$, $T/P\cong R/(R\cap P)$ is a domain.
\item $T/R$ is a left and right singular $R$-module.
\end{enumerate}
\end{thm}
\begin{proof}
Let $x\in Z({}_RT)$, then we prove that $x\in Z(T_R)$. To see this we must show that $ann_R(x)$ is essential as right ideal of $R$ (note that $l.ann_R(x)=r.ann_R(x)$ for $T$ is a symmetric ring). Hence assume that $0\neq a\in R$. Since $ann_R(x)$ is essential as left ideal in $R$, we deduce that there exists $r\in R$, such that $0\neq ra\in ann_R(x)$. Therefore $rax=0$ and $ra\neq 0$. Consequently, $ar\neq 0$, for $T$ is a symmetric ring and therefore $xar=0$. This immediately shows that $ann_R(x)$ is essential as right ideal. Hence $x\in Z(T_R)$. Similarly, $Z(T_R)\subseteq Z({}_RT)$ and hence $Z(T_R)=Z({}_RT)$ is a proper ideal of $T$, by the previous lemma. Now assume that $Z\neq 0$. Therefore $R\cap Z=Z(R)=0$ (for $R$ is reduced) and thus $R\oplus Z=T$, by the maximality of $R$. Thus $Z$ is a minimal ideal of $T$ by $(1)$ Theorem \ref{t1}. Now note that for each $0\neq x\in Z$, $RxT$ and $TxR$ are nonzero $(R,T)$-subbimodule and $(T,R)$-subbimodule of $Z$, respectively. Thus by the maximality of $R$, we conclude that $R+RxT=T=R+TxR$. From these equalities we immediately deduce that $Z=RxT=TxR$. Therefore $(1)$ holds. Since $T$ is reduced, we infer that $Z^2\neq 0$. Thus by $(2)$ and $(3)$ of Theorem \ref{t1}, we conclude that $(R:T)=(R:T)_l=(R:T)_r=ann_R(Z)$. Clearly, $Soc(R)Z=ZSoc(R)=0$, hence $Soc(R)\subseteq (R:T)$. By $(4)$ of Theorem \ref{t1}, we deduce that $Q=ann_T(Z)$ is a minimal prime ideal of $T$ and $Q\cap R=(R:T)$. Therefore $(2)$ is proved. For $(3)$, note that $T/Z\cong R$ as ring, and therefore $T/Z$ is a reduced ring. Thus by \cite[Lemma 12.6]{lam}, we conclude that each minimal prime ideal of $T/Z$, is a completely prime ideal, i.e., each minimal prime ideal $P$ over $Z$ in $T$ is a completely prime ideal (i.e., $T/P$ is a domain). Clearly, $P\nsubseteq R$, for $Z\nsubseteq R$. Thus $R+P=T$, by the maximality of $R$. Thus we conclude $T/P\cong R/(R\cap P)$ as rings, therefore the final part of $(3)$ is proved. Finally for $(4)$, from $R\oplus Z=T$, we deduce that $T/R\cong Z$ as left and right $R$-modules (note that all of the part of the equality are left and right $R$-submodules of $T$). Since $Z$ is left and right singular $R$-module, we deduce that $T/R$ is left and right singular $R$-module.
\end{proof}

In the following example we give an example of a maximal subring $R$ of a reduced ring $T$ with $Z({}_RT)\neq 0$.

\begin{exm}\label{t13}
Let $T=\mathbb{Z}\times\mathbb{Z}_2$ and $R=\{(n,\bar{n})\ |\ n\in\mathbb{Z}\}\cong \mathbb{Z}$ be the prime subring of $T$. It is clear that $R$ is a maximal subring of $T$ and $T$ is a reduced ring. One can easily see that $ann_R((0,\bar{1}))=2R$ which is essential in $R$. Hence $Z_R(T)\neq 0$.
\end{exm}

In what following we want to study Jacobson radicals, upper nilradicals and lower nilradicals of the ring extension $R\subseteq T$ where $R$ is a maximal subring of $T$. Specially when $T=R\oplus I$, where $I$ is a nonzero ideal of $T$. First we need the following fact.

\begin{prop}\label{t17}
Let $R$ be a maximal subring of a ring $T$. Then at least one of the following holds:
\begin{enumerate}
\item $(R:T)$ is the unique maximal ideal of $T$.
\item $J(T_R)$ and $J({}_RT)$ are proper submodules of $T$ and for each left (resp. right) $R$-submodule $A$ of $T$, either $A\subseteq J({}_RT)$ (resp. $J(T_R)$) or $A$ has a maximal left (resp. right) $R$-submodule.
\end{enumerate}
\end{prop}
\begin{proof}
Assume that $M$ is a maximal ideal of $T$ which is not contained in $R$, then by maximality of $R$ we obtain that $R+M=T$. Therefore $R/(R\cap M)\cong T/M$ as left and right $R$-modules. This immediately implies that $T/M$ has a maximal left (resp. right) $R$-submodule and thus $T$ has a maximal left (resp. right) $R$-submodule. Thus $J({}_RT)$ and $J(T_R)$ are proper submodules of $T$. Therefore, if $J({}_RT)=T$ (or $J(T_R)=T$), we conclude that each maximal ideal of $T$ is contained in $R$ and hence is equal to $(R:T)$. Thus $(R:T)$ is the unique maximal ideal of $T$. For the final part of $(2)$, if $A$ is a left $R$-submodule of $T$ which is not contained in $J({}_RT)$, then there exists a maximal left $R$-submodule $N$ of $T$ such that $A$ is not contained in $N$. Therefore $A+N=T$ and thus $A/(A\cap N)\cong T/N$ as left $R$-modules. Hence $A\cap N$ is a maximal submodule of $A$ and we are done.
\end{proof}

Now we have the following result for Jacobson radicals.

\begin{thm}\label{t18}
Let $R$ be a maximal subring of a ring $T$ and $T=R\oplus I$, where $I$ is an ideal of $T$. Then either $J(T)\cap R=J(R)$ or the following hold:
\begin{enumerate}
\item $I^2=I$.
\item $J(T)\subseteq (R:T)=(R:T)_l=(R:T)_r$.
\item $TJ(R)=J(R)T\nsubseteq R$; in particular, $(R:T)$ is not a left/right primitive ideal of $R$ and hence $dim(R)\geq 1$. Moreover, $IJ(R)=J(R)I=I$ and ${}_RI$, $I_R$ has no maximal submodules.
\item  $J({}_RT)$ and $J(T_R)$ are proper submodules of $T$ and contain $I$.
\item $(R:T)=l.ann_T(I)$ is a primitive and a minimal prime ideal of $T$.
\item $(R:T)$ is not a maximal ideal of $T$. Moreover, $(R:T)\oplus I$ is a prime ideal of $T$ which is not maximal. In particular, $dim(T)\geq 2$.
\item $I$ is contained in each maximal ideal of $T$.
\item $I$ is unique, i.e., if $I_1$ is a nonzero proper ideal of $T$ such that $R\cap I_1=0$, then $I_1=I$.
\end{enumerate}
\end{thm}
\begin{proof}
First we prove that $U(R)=U(T)\cap R$ and $J(T)\cap R\subseteq J(R)$. Assume that $x\in U(T)\cap R$, thus $Tx=T=xT$ and therefore by $(6)$ of Theorem \ref{t1}, we deduce that $Rx=Tx\cap R=T\cap R=R$ and similarly $xR=R$, which immediately implies that $x\in U(R)$. Thus $U(R)=U(T)\cap R$. Now suppose that $x\in J(T)\cap R$ and $y\in R$. Thus $1-xy\in U(T)\cap R=U(R)$ and therefore $x\in J(R)$. Now, if for each left primitive ideal $Q$ of $T$, we have $Q\nsubseteq R$, then by maximality of $R$, we deduce that $R+Q=T$ and therefore $R/(R\cap Q)\cong T/Q$ as rings. Hence $R\cap Q$ is a left primitive ideal of $R$ and therefore $J(R)\subseteq R\cap Q$. Hence $J(R)\subseteq J(T)\cap R$, and thus $J(T)\cap R=J(R)$. Thus we may assume that there exist a left primitive ideal $Q_1$ and (similarly) a right primitive ideal $Q_2$ of $T$ such that $Q_1$ and $Q_2$ are contained in $R$ and therefore $Q_1,Q_2\subseteq (R:T)$. In particular $J(T)\subseteq R$ and therefore $J(T)\subseteq J(R)$. Since for each $i$, $i=1,2$, $Q_i$ is a prime ideal of $T$, we deduce that $I^2\neq 0$, for otherwise $I\subseteq Q_i\subseteq R$ which is absurd. By $(1)$ of Theorem \ref{t1}, $I$ is minimal and therefore we deduce that $I^2=I$. Thus by $(3)$ of Theorem \ref{t1}, we conclude that $P:=(R:T)=(R:T)_l=(R:T)_r$. Hence $J(T)\subseteq Q_i\subseteq P$. Now if $J(R)\subseteq P$ (and therefore $J(R)T$ and $TJ(R)$ are contained in $R$), then since $T=R\oplus I$, we deduce that $T/R\cong I$ as left and right $R$-module and therefore $J(R)I=0=IJ(R)$. Now assume that $x\in J(R)$, we claim that $x\in J(T)$. To see this, let $t=r+i\in T=R\oplus I$, where $r\in R$ and $i\in I$, then $1-tx=1-(r+i)x=1-rx\in U(R)\subseteq U(T)$, for $ix\in IJ(R)=0$. Therefore $x\in J(T)$ and thus $J(R)=J(T)$ and hence we are done. Thus we may assume that $J(R)\nsubseteq P$ (and therefore $TJ(R)$ and $J(R)T$ are note contained in $R$; also note that this immediately implies that $P$ is not a one-sided primitive ideal of $R$ and hence is not a maximal ideal of $R$ too, thus $dim(R)\geq 1$), hence $J(R)I$ and $IJ(R)$ are nonzero. Since $R$ is a maximal subring of $T$, we conclude that $R+J(R)I=T=R+IJ(R)$ and since $IJ(R), J(R)I\subseteq I$, $I\cap R=0$ and $R\oplus I=T$, we deduce that $I=IJ(R)=J(R)I$. Thus by $(7)$ of Theorem \ref{t1}, we have $TJ(R)=J(R)T$. The final part of $(3)$ is an immediate consequence of Nakayama's Lemma. Clearly, since $T/I\cong R$ as left and right $R$-modules, we conclude that ${}_RT$ and $T_R$ have maximal $R$-submodules and therefore the first part of $(4)$ holds. Now if $K$ is a maximal left $R$-submodule of $T$ which does not contain $I$, then $K+I=T$. Therefore $I/(I\cap K)\cong T/K$, as left $R$-module, which means $I\cap K$ is a maximal left $R$-submodule of $I$ which is a contradiction by the final part of $(3)$. Thus $I\subseteq J({}_RT)$ and similarly $J(T_R)$ contains $I$. For $(5)$, let $A=l.ann_T(I)$, then by minimality of $I$, if $A\cap I=I$, then $I^2=0$ which is absurd by $(1)$. Thus we conclude that $A\cap I=0$ and therefore $IA=0$. Hence $A\subseteq r.ann_T(I)$. Similarly, if $B=r.ann_T(I)$, we have $BI=0$ and therefore $B\subseteq A$. Hence $A=B$, i.e., $l.ann_T(I)=r.ann_T(I)$. Now note that $AI=0\subseteq Q_i$ and since $Q_i$ does not contain $I$, we deduce that $A\subseteq Q_i\subseteq (R:T)$. Hence $A=Q_i=P$ (note that $A=l.ann_R(I)=(R:T)_r=(R:T)=(R:T)_l=r.ann_R(I)$). Hence $A=Q_1=Q_2=P$ is a primitive ideal of $T$ which is also a prime ideal of $R$. If $Q$ is a minimal ideal of $T$, which properly contained in $A$, then $AI=0\subseteq Q$ implies that $I\subseteq Q\subseteq A$ which is impossible. Thus $A$ is a minimal prime ideal of $T$. For $(6)$, if $P=(R:T)=A$ is a maximal ideal of $T$, then $I+P=T$. Clearly $I\cap P=0$, and therefore we have the ring isomorphism $\alpha: T\longrightarrow T/I\times T/P$ which is defined by $\alpha(t)=(t+I,t+P)$. Since $T/I\cong R$ as ring, let $\beta: T/I\times T/P\longrightarrow R\times T/P$ be the ring isomorphism defined by $\beta(t+I,s+P)=(r,s+P)$, where $t=r+i$, $r\in R$ and $i\in I$. Then clearly $f=\beta\alpha$ is a ring isomorphism and $f(R)=\{(r, r+P)\ |\ r\in R\}$ is a maximal subring of $Im(f)=R\times T/P$, for $R$ is a maximal subring of $T$. But it is obvious that $R\times R/P$ is a proper subring of $T$ (note $R/P$ is a maximal subring of $T/P$) which properly contains $f(R)$ (note $(0, 1+P)\in R\times R/P$ but not in $f(R)$). Thus $P$ is not a maximal ideal of $T$. The final part of $(6)$ is evident by the second part of $(3)$ and Remark \ref{t6}, in particular $dim(T)\geq 2$. Now for $(7)$, by $(6)$ note that for each maximal ideal $M$ of $T$ we infer that $R$ does not contain $M$ and therefore by maximality of $R$, $R+M=T$. Thus $R/(R\cap M)\cong T/M$ as left (and right) $R$-modules. In particular, $T/M$ has a maximal left $R$-submodule. Now if $I$ is not contained in $M$, then $I\cap M=0$ by minimality of $I$ and $I+M=T$ by maximality of $M$. Therefore, $T/M\cong I$ as left (and right) $R$-modules. Hence $I\cong R/(R\cap M)$ as left (and right) $R$-modules. Thus $I$ has a maximal left (and right) $R$-submodules which is impossible by $(3)$. Hence $I$ is contained in each maximal ideal $M$ of $T$ and therefore $(7)$ is true. Finally for $(8)$, assume that $I_1\neq I$ be a proper nonzero ideal of $T$ such that $R\cap I_1=0$. Therefore $R\oplus I_1=T$ and hence $I_1$ is a minimal ideal of $T$, by $(1)$ of Theorem \ref{t1}. Since $I\neq I_1$ are minimal ideals of $T$, we conclude that $I\cap I_1=0$. Therefore $II_1=I_1I=0$. Hence $I_1\subseteq A\subseteq R$, which is absurd. Thus $(8)$ holds.
\end{proof}

Similar to the previous result we have the following fact for upper nilradical.

\begin{thm}\label{t20}
Let $R$ be a maximal subring of a ring $T$. Then one of the following condition holds:
\begin{itemize}
\item[$(i)$] There exists a strongly prime ideal $Q$ of $T$ such that $Q\subseteq R$. In particular, $Nil^*(T)\subseteq Q\subseteq (R:T)$.
\item[$(ii)$] $Nil^*(T)\cap R=Nil^*(R)$. In particular, either $Nil^*(T)=Nil^*(R)$ or $R/Nil^*(R)\cong T/Nil^*(T)$, as rings.
\end{itemize}
Moreover, if there exists a nonzero ideal $I$ of $T$ such that $R\cap I=0$, then either $Nil^*(T)\cap R=Nil^*(R)$, or the following hold:
\begin{enumerate}
\item $I^2=I$.
\item $(R:T)=(R:T)_l=(R:T)_r=ann_T(I)$ is a strongly prime ideal of $T$ which is a minimal prime ideal of $T$.
\item $Nil^*(R)\nsubseteq (R:T)$. In particular, $IN=NI$ and $TN=NT$, where $N=Nil^*(R)$.
\item $(R:T)\oplus I$ is a prime ideal of $T$ which is not a strongly prime ideal of $T$. In particular, $(R:T)\oplus I$ is not a maximal ideal of $T$ and $dim(T)\geq 2$.
\item $I$ is unique, i.e., if $I_1$ is a nonzero ideal of $T$ such that $R\cap I_1=0$, then $I_1=I$.
\end{enumerate}
\end{thm}
\begin{proof}
If there exists a strongly prime ideal $Q$ of $T$ such that $Q\subseteq R$, then $(i)$ holds. Hence assume that for each strongly prime ideal $Q$ of $T$ we have $Q\nsubseteq R$ and therefore by maximality of $R$, we conclude that $R+Q=T$. Hence $R/(R\cap Q)\cong T/Q$ as rings, which implies that $R\cap Q$ is a strongly prime ideal of $R$. Therefore if $\{Q_\alpha\}_{\alpha\in\Gamma}$ is the set of all strongly prime ideals of $T$, we conclude that $Nil^*(R)\subseteq \bigcap_{\alpha\in\Gamma} (R\cap Q_{\alpha})=Nil^*(T)\cap R$.
Now note that, if $x\in Nil^*(T)\cap R$, then $TxT$ is a nil ideal of $T$ and since $RxR\subseteq TxT$, we immediately deduce that $RxR$ is a nil ideal of $R$ and hence $Nil^*(T)\cap R\subseteq Nil^*(R)$. Therefore $Nil^*(T)\cap R=Nil^*(R)$. Hence the first part of $(ii)$ holds. Now if $Nil^*(T)\subseteq R$, then $Nil^*(T)=Nil^*(R)$ and if $Nil^*(T)\nsubseteq R$, then by maximality of $R$ we conclude that $R+Nil^*(T)=T$. Therefore $R/Nil^*(R)=R/(R\cap Nil^*(T))\cong T/Nil^*(T)$, as rings. The proofs of the case $T=R\oplus I$, is very similar to the proof of the previous theorem and hence remain to the reader.
\end{proof}

Next we have the following main result for lower nilradical, fortunately as we see in the next result, the equality $Nil_*(R)=Nil_*(T)\cap R$ holds for a maximal subring $R$ of a ring $T$ with assumption $T=R\oplus I$, for an ideal $I$ of $T$.

\begin{thm}\label{t21}
Let $R$ be a maximal subring of a ring $T$. Then at least one of the following holds:
\begin{enumerate}
\item There exists a (minimal) prime ideal $Q$ of $T$ such that $Q\subseteq (R:T)$. In particular, $Nil_*(T)\subseteq Q\subseteq (R:T)$.
\item $Nil_*(R)\subseteq Nil_*(T)\cap R\subseteq Nil^*(R)$.
\end{enumerate}
In particular, if there exists a nonzero ideal $I$ of $T$ such that $R\cap I=0$, then $Nil_*(R)=Nil_*(T)\cap R$.
\end{thm}
\begin{proof}
It is clear that if $R$ contains a prime ideal of $T$, then $(1)$ holds. Thus assume that $R$ does not contain any prime ideal $Q$ of $T$. Hence $R+Q=T$ by maximality of $R$ and therefore $R/(R\cap Q)\cong T/Q$ as rings. Thus $R\cap Q$ is a prime ideal of $R$. Thus $Nil_*(T)\cap R=\bigcap_{Q\in Spec(T)} (R\cap Q)$ is a semiprime ideal of $R$ and therefore contains $Nil_*(R)$. Since $Nil_*(T)$ is a nil ideal of $T$, we infer that $Nil_*(T)\cap R$ is nil ideal of $R$ and therefore is contained in $Nil^*(R)$, hence $(2)$ holds. Now assume that $I$ be a nonzero ideal of $T$ such that $R\cap I=0$ and therefore $T=R\oplus I$, by maximality of $R$. By $(1)$ of Theorem \ref{t1}, $I$ is minimal and hence we have two cases: $(i)$ $I^2=0$ (or whenever $I\subseteq Nil_*(T)$). In this case first we claim that $Nil_*(T)=Nil_*(R)\oplus I$. To see this first note that since $I^2=0$, we deduce that for each prime ideal $Q$ of $T$, we have $I\subseteq Q$ and therefore $Q\nsubseteq R$. Hence by $(2)$, $Nil_*(R)\subseteq Nil_*(T)$. Therefore $Nil_*(R)\oplus I\subseteq Nil_*(T)$. Now since $T/(Nil_*(R)\oplus I)\cong R/Nil_*(R)$ as rings, we deduce that $Nil_*(R)\oplus I$ is a semiprime ideal of $T$ and therefore contains $Nil_*(T)$. Thus $Nil_*(R)\oplus I=Nil_*(T)$. Finally in this case we must show that $Nil_*(T)\cap R\subseteq Nil_*(R)$. Let $x\in Nil_*(T)\cap R$, thus $x=r+i$, where $r\in Nil_*(R)$ and $i\in I$, for $Nil_*(T)=Nil_*(R)\oplus I$. Thus $x-r=i\in I\cap R=0$ and therefore $x=r\in Nil_*(R)$. Thus the equality holds in this case. $(ii)$ Hence assume that $I^2=I$ and as we mentioned in the previous case we may also assume that $I\nsubseteq Nil_*(T)$, which immediately implies that $I\cap Nil_*(T)=0$, for $I$ is a minimal ideal of $T$. Now similar to the previous case if for each prime ideal $Q$ of $T$, $R$ does not contain $Q$, then by $(2)$, we deduce that $Nil_*(R)\subseteq Nil_*(T)\cap R$. Also note that as we see in case $(i)$, $Nil_*(R)\oplus I$ is a semiprime ideal of $T$ and hence contains $Nil_*(T)$. Now we prove that $Nil_*(T)=Nil_*(R)$. Assume that $x\in Nil_*(T)$, thus $x=r+i$, where $r\in Nil_*(R)$ and $i\in I$. Therefore $x-r=i\in Nil_*(T)\cap I=0$ and therefore $x=r\in Nil_*(T)$. Thus $Nil_*(T)=Nil_*(R)$. Hence assume that there exists a prime ideal $Q$ of $T$ such that $Q\subseteq R$. Thus $Q\subseteq (R:T)$. Since $I^2=I$, we deduce that $(R:T)=(R:T)_l=(R:T)_r$, by $(3)$ of Theorem \ref{t1}. Therefore $(R:T)$ is a prime ideal of $R$ and hence $Nil_*(R)\subseteq (R:T)$. Since $Q$ is a prime ideal of $T$, we also deduce that $Nil_*(T)\subseteq Q\subseteq (R:T)$. Since $Q$ is a prime ideal of $T$ which is contained in $R$ and $I\cap R=0$, we immediately conclude that $I\cap Q=0$, by minimality of $I$. Hence $QI=IQ=0$, and therefore $Q=l.ann_T(I)=r.ann_T(I)=ann_T(I)$. This immediate implies that $Q$ is a minimal prime ideal of $T$. Now since $Q=ann_T(I)$ is contained in $R$, we deduce that $Q=(R:T)$, for $(R:T)=(R:T)_l=r.ann_R(I)=r.ann_T(I)=Q$, and therefore $Nil_*(R)\subseteq Q$. Now note that similar to the previous, $Nil_*(T)\subseteq Nil_*(R)\oplus I$ and $Nil_*(T)\subseteq R$, imply that $Nil_*(T)\subseteq Nil_*(R)$. On the other hand for each minimal prime ideal $Q'\neq Q$ of $T$, we deduce that $Q'\nsubseteq R$, for $Q=(R:T)$ is a minimal prime ideal of $T$. Therefore $R+Q'=T$ by maximality of $R$ and thus $R/(R\cap Q')\cong T/Q'$, as rings. Hence $R\cap Q'$ is a prime ideal of $R$. Therefore $Nil_*(R)\subseteq Q'$. Hence $Nil_*(R)$ is contained in each minimal prime ideal of $T$ and therefore is contained in $Nil_*(T)$. Thus $Nil_*(T)=Nil_*(R)$.
\end{proof}

\begin{rem}\label{t22}
Similar to the proof of the previous theorem one can easily see that if $R$ is a maximal subring of a ring $T$, $T=R\oplus I$, where $I$ is an ideal of $T$ and $I\subseteq J(T)$ (resp. $I\subseteq Nil^*(T)$), then $J(T)=J(R)\oplus I$ (resp. $Nil^*(T)=Nil^*(R)\oplus I$).
\end{rem}

In the next two propositions we study the socle and singularity for the ring extension $R\subseteq T$, where $R$ is a maximal subring of $T$ and $T$ is of the form $R\oplus I$, where $I$ is an ideal of $T$.

\begin{prop}\label{t23}
Let $R$ be a maximal subring of a ring $T$ and $I$ be a nonzero ideal of $T$ such that $R\cap I=0$. Then either $Soc({}_RR)=Soc({}_RT)\subseteq (R:T)_r$ or the following hold:
\begin{enumerate}
\item $Soc({}_RT)=Soc({}_RR)\oplus Soc({}_RI)$.
\item $Soc({}_RI)=I$, i.e., ${}_RI$ is a semisimple left $R$-module.
\item All simple $R$-submodules of ${}_R I$ are isomorphic as a left $R$-module. In fact ${}_RI\cong \oplus_{\alpha\in \Gamma}S$, as left $R$-module, for a simple left $R$-submodule of $I$, and a set $\Gamma$.
\item $(R:T)_r$ is a left primitive ideal of $R$. In particular, $J(R)\subseteq (R:T)_r$.
\item $J(R)=J(T)\cap R$ and $Nil^*(R)=Nil^*(T)\cap R$.
\item if in addition, $I^2=0$, then $I=RaR$, for each $0\neq a\in I$.
\end{enumerate}
\end{prop}
\begin{proof}
Since $T=R\oplus I$, we conclude that $Soc({}_RT)=Soc({}_RR)\oplus Soc({}_RI)$. Hence if $Soc({}_RI)=0$, then $Soc({}_RR)=Soc({}_RT)$. Now note that by the proof of $(3)$ in Lemma \ref{t8},  $Soc({}_RT)$ is a right ideal of $T$, which is contained in $R$, therefore $Soc({}_RR)=Soc({}_RT)\subseteq (R:T)_r$. Hence assume that $Soc({}_RI)\neq 0$. Let $S$ be a simple left $R$-submodule of $I$. Therefore $ST$ is a $(R,T)$-subbimodule of $I$, which is not contained in $R$. Thus $R\oplus ST=T$, by maximality of $R$. Since $ST\subseteq I$ and $R\oplus I=T$, we immediately conclude that $I=ST$. Hence $I=ST=\sum_{t\in T}St$ and for each $t\in T$, note that either $St=0$ or $St\cong S$ as a left $R$-module. Hence $I=ST$ is a semisimple left $R$-module, all of simple left $R$-submodule of $I$ are isomorphic to $S$, i.e., $Soc({}_RI)=I$ and there exists $T'\subseteq T$ such that $I\cong\bigoplus_{t\in T'} S$, as left $R$-modules. Hence $T/R\cong I\cong \bigoplus_{t\in T'} S$ and therefore $(R:T)_r=l.ann_R(S)$ is a left primitive ideal of $R$. $(5)$ is an immediate consequences of $(3)$ of Theorem \ref{t18}, $(3)$ of Theorem \ref{t20}, the facts that ${}_RI$ has a maximal (left) $R$-submodule in this case, and also $J(R)I=0$. Finally for $(6)$, if $I^2=0$, then for each nonzero $a\in I$, one can easily check that $R\oplus RaR$ is a subring of $T$ which properly contains $R$ and therefore $R\oplus RaR=T$, by maximality of $R$. Now since $RaR\subseteq I$ and $R\oplus RaR=T=R\oplus I$, we conclude that $I=RaR$.
\end{proof}

\begin{prop}\label{t24}
Let $R$ be a maximal subring of a ring $T$ and $I$ be a nonzero ideal of $T$ such that $R\cap I=0$. Then either $Z({}_RT)=Z({}_RR)\subseteq (R:T)_r$ or the following hold:
\begin{enumerate}
\item $Z({}_RT)=Z({}_RR)\oplus Z({}_RI)$.
\item $Z({}_RI)=I$, i.e., ${}_RI$ is a singular left $R$-submodule of $T$ (i.e., ${}_R(T/R)$ is a singular module).
\item $Soc({}_RR)\subseteq (R:T)_r$.
\end{enumerate}
\end{prop}
\begin{proof}
First note that $T=R\oplus I$ implies that $Z({}_RT)=Z({}_RR)\oplus Z({}_RI)$. Hence if $Z({}_RI)=0$, then $Z({}_RR)=Z({}_RT)$ and since $Z({}_RT)$ is a right ideal of $T$, by Lemma \ref{t11}, we deduce that $Z({}_RT)=Z({}_RR)\subseteq (R:T)_r$. Thus assume that $Z({}_RI)\neq 0$. Since $I$ is an ideal of $T$ we deduce that $Z({}_RI)$ is a right ideal of $T$ which is contained in $I$, and clearly is a left nonzero $R$-submodule of $T$, therefore $R\oplus Z({}_RI)=T$. Hence by $R\oplus Z({}_RI)=T=R\oplus I$ and $Z({}_RI)\subseteq I$ we deduce that $Z({}_RI)=I$, i.e., $I$ is a singular left $R$-submodule of $T$. Finally note that $Soc({}_RR)Z({}_RI)=0$ implies that $Soc({}_RR)I=0$ and since $T/R\cong I$ as left $R$-modules, we conclude that $Soc({}_RR)T\subseteq R$. Therefore $Soc({}_RR)\subseteq (R:T)_r$.
\end{proof}

Next we investigate about the relation between lower/upper nilradical and Jacobson radical of $R$ and $T$, where $R$ is a maximal subring of a ring $T$ and either $R$ or $T$ satisfies ACC or DCC.

\begin{prop}\label{t25}
Let $R$ be a maximal subring of a ring $T$ and $I$ be a nonzero ideal of $T$ with $R\cap I=0$. Then the following hold:
\begin{enumerate}
\item If either $R$ or $T$ is a one-sided Artinian ring, then $J(R)=J(T)\cap R$ and $Nil^*(R)=Nil^*(T)\cap R$.
\item If $R$ is left Artinian ring, then $Soc({}_RI)=I$. In particular, ${}_R T$ is Artinian if and only if ${}_RI$ is finitely generated.
\end{enumerate}
\end{prop}
\begin{proof}
$(1)$ First assume that $T$ is a one-sided Artinian ring, therefore $dim(T)=0$. Hence by $(6)$ of Theorem \ref{t18}, we deduce that $J(T)\cap R=J(R)$; and by $(4)$ of Theorem \ref{t20}, we conclude that $Nil^*(T)\cap R=Nil^*(R)$. Now assume that $R$ is a one-sided Artinian ring, thus $dim(R)=0$. Hence by $(3)$ of Theorem \ref{t18}, we have $J(T)\cap R=J(R)$. Finally note that by Theorem \ref{pt1}, $(R:T)_l$ is a prime ideal of $R$, thus by $(2)$ and $(3)$ of Theorem \ref{t20}, we conclude that $Nil^*(T)\cap R=Nil^*(R)$. For $(2)$, since $R$ is a left Artinian ring, we deduce that $Soc({}_RI)\neq 0$. Therefore by $(2)$ of Proposition \ref{t23}, we conclude that $Soc({}_RI)=I$. Consequently, ${}_RT$ is Artinian if and only if ${}_RI$ is finitely generated and if any of these equivalent conditions holds then $T$ is a left Artinian ring.
\end{proof}

If $S$ is a one-sided Noetherian ring, it is well known that the set of minimal prime ideals of $S$ is finite and $Nil^*(S)=Nil_*(S)$ is a nilpotent ideal of $S$ and every one-sided nil ideal of $S$ is nilpotent, see \cite{lam} and \cite{rvn}. Moreover, if $R$ is a subring of $S$, then $ACC$ holds on the set of one-sided annihilators of $R$, see \cite{lam2}, and thus $Nil^*(R)=Nil_*(R)$, see \cite{lam}.

\begin{cor}\label{t26}
Let $T$ be a left Noetherian ring and $R$ be a maximal subring of $T$. Then either $Nil^*(T)\subseteq Nil_*(R)$ or $Nil^*(T)\cap R=Nil^*(R)=Nil_*(R)$ is a nilpotent ideal.
\end{cor}
\begin{proof}
Since $T$ is a left Noetherian ring, we infer that $Nil^*(T)$, and therefore $Nil^*(T)\cap R$, are nilpotent. Thus $Nil^*(T)\cap R\subseteq Nil_*(R)$. Hence, if $Nil^*(T)\subseteq R$, then $Nil^*(T)\subseteq Nil_*(R)$. Thus assume that $Nil^*(T)\nsubseteq R$. Therefore $R+Nil^*(T)=T$, by maximality of $R$ and thus $R/(R\cap Nil^*(T))\cong T/Nil^*(T)$ as rings. Hence $R/(R\cap Nil^*(T))$ has no nil ideal and therefore $Nil^*(R)\subseteq Nil^*(T)\cap R$. Since $Nil^*(T)\cap R\subseteq Nil_*(R)\subseteq Nil^*(R)$, we conclude that $Nil^*(T)\cap R= Nil^*(R)=Nil_*(R)$ is nilpotent.
\end{proof}

\begin{rem}\label{t27}
If $T$ is a left Artinian ring, then each prime ideal of $T$ is maximal and thus minimal prime ideals, prime ideals, left/right primitive ideals and maximal ideals of $T$ are coincided,  and $T$ has only finitely many maximal ideals, say  $P_1,\ldots,P_n$. Hence $J(T)=Nil^*(T)=Nil_*(T)=\bigcap_{i=1}^nP_i$ is a nilpotent ideal of $T$, see \cite{rvn}.
\end{rem}

\begin{thm}\label{t28}
Let $T$ be a left Artinian ring and $R$ be a maximal subring of $T$. Then exactly one of the following holds:
\begin{itemize}
\item[$(1)$] There exists a (unique) maximal ideal $M$ of $T$ such that $M \subseteq R$. In particular, $R/M$ is a maximal subring of Artinian simple ring $T/M$ (which is isomorphic to $\mathbb{M}_n(D)$, for some natural number $n$ and a division ring $D$). Moreover, $(R:T)=M$ and each prime ideal of $R$ either contains $(R:T)=M$ or is a maximal ideal of $R$.
\item[$(2)$] $R$ is a zero-dimensional ring, $Spec(R)$ is finite ($|Spec(R)|\leq |Max(T)|$) and $J(R)=Nil^*(R)=Nil_*(R)=J(T)\cap R$ is a nilpotent ideal of $R$ and one of the following holds:
\begin{itemize}
\item[$(a)$] There exist maximal ideals $M\neq N$ of $T$ such that $T/M\cong T/N$ as rings.
\item[$(b)$] $R/J(R)\cong T/J(T)$ as rings. In particular, $R$ is a semiprimary ring.
\end{itemize}
\end{itemize}
\end{thm}
\begin{proof}
First assume that there exists a maximal ideal $M$ of $T$ such that $M\subseteq R$. It is clear that $M$ must be unique for $R$ is a proper subring of $T$ and also $(R:T)=M$. Now let $P$ be a prime ideal of $R$, then $R\setminus P$ is a $m$-system in $T$ and therefore there exists a prime (maximal) ideal $N$ of $T$ such that $(R\setminus P)\cap N=0$. Thus $R\cap N\subseteq P$. Now we have two cases: $(i)$ $N\subseteq R$ and therefore $N=M=(R:T)\subseteq P$. $(ii)$ $N\nsubseteq R$, therefore $R+N=T$, by the maximality of $R$. Hence $R/(R\cap N)\cong T/N$, as rings. Thus $R\cap N$ is a maximal ideal of $R$, for $T/N$ is a simple ring. Therefore $P=R\cap N$ is a maximal ideal of $R$. Now assume that $R$ does not contain any maximal ideal of $T$ and let $Max(T)=\{M_1,\ldots,M_n\}$. Similar to the latter proof for each $i$, $R\cap M_i$ is a maximal ideal of $R$ and each prime ideal of $R$ is of the form $R\cap M_i$ for some $i$. Therefore $R$ is a zero-dimensional ring and $|Spec(R)|\leq |Max(T)|$. Thus $J(R)=Nil^*(R)=Nil_*(R)\subseteq J(T)\cap R$. Now note that $J(T)\cap R$ is a nilpotent ideal of $R$ and therefore $J(T)\cap R\subseteq Nil_*(R)$. Hence $J(R)=Nil^*(R)=Nil_*(R)=J(T)\cap R$ is nilpotent. Now we have two cases: $(a)$ $J(T)\subseteq R$, thus $R/J(T)$ is a maximal subring of $T/J(T)=T/M_1\times\cdots\times T/M_n$. Therefore by Theorem \ref{pt1} and the fact that $R$ does not contain any maximal ideals of $T$, we deduce that there exist distinct maximal ideals $M_i$ and $M_j$ of $T$, $i\neq j$ such that $T/M_i\cong T/M_j$ as rings, and $R/J(T)\cong \prod_{i\neq k=1}^n T/M_k$. $(b)$ $J(T)\nsubseteq R$, therefore $R+J(T)=T$, by the maximality of $R$. Hence $R/J(R)=R/(J(T)\cap R)\cong T/J(T)$, as rings, and thus $R/J(R)$ is an Artinian ring and therefore $R$ is a semiprimary ring.
\end{proof}

\section{left Artinian maximal subrings}

In this section we want to study relation between lower/upper nilradicals and Jacobson radicals of $R$ and $T$, where $R$ is a maximal subring of a ring $T$ and $R$ is a left Artinian ring. First we need the following well known remark about Artinian subring of a ring and the next two lemmas.

\begin{rem}\label{t29}
Let $R$ be a left Artinian subring of a ring $T$, then one can easily see that $U(R)=U(T)\cap R$ and therefore $J(T)\cap R\subseteq J(R)$.
\end{rem}

Now the following is in order.

\begin{lem}\label{t30}
Let $R$ be a maximal subring of a ring $T$ and $R$ be a left Artinian ring. Then the following hold:
\begin{enumerate}
\item For each prime ideal $P$ of $T$, either $P\subseteq R$ (i.e., $P\subseteq (R:T)$) or $P$ is a maximal ideal of $T$.
\item Either $J(T)\subseteq J(R)$ or $J(R)=J(T)\cap R$.
\item If each left (resp. right) primitive ideal of $T$ is not contained in $R$, then $J(R)=J(T)\cap R$ and each left (resp. right) primitive ideal of $T$ is maximal.
\item Either $Nil^*(T)\subseteq J(R)$ or $Nil^*(T)\cap R=J(R)=J(T)\cap R$.
\item If each strongly prime ideal of $T$ is not contained in $R$, then $Nil^*(T)\cap R=J(R)=J(T)\cap R$ and each strongly prime ideal of $T$ is maximal.
\item Either $Nil_*(T)\subseteq J(R)$ or $Nil_*(T)\cap R=J(R)=Nil^*(T)\cap R=J(T)\cap R$.
\item If each prime ideal of $T$ is not contained in $R$, then $Nil_*(T)\cap R=J(R)=Nil^*(T)\cap R=J(T)\cap R$ and $dim(T)=0$.
\item DCC holds on $Spec(T)$.
\item Each nonzero left $R$-submodule of $T$ contains a simple left $R$-submodule. In fact, $Soc({}_RT)$ is an essential left $R$-submodule of $T$. Moreover, if there exists a nonzero ideal $I$ of $T$ such that $I\cap R=0$, then $(1)-(6)$ of Proposition \ref{t23} hold.
\end{enumerate}
\end{lem}
\begin{proof}
$(1)$ Let $P$ be a prime ideal of $T$ which is not contained in $R$, then by maximality of $R$ we conclude that $R+P=T$ and therefore $R/(R\cap P)\cong T/P$, as rings. Hence $P\cap R$ is a prime ideal of $R$ and since $R$ is a left Artinian ring we conclude that $P\cap R$ is a maximal ideal of $R$ and therefore $P$ is a maximal ideal of $T$. $(2)$ If $J(T)\subseteq R$, then it is clear that $J(T)\subseteq J(R)$, by the previous remark. Hence assume that $J(T)\nsubseteq R$ and therefore by maximality of $R$ we conclude that $R+J(T)=T$. Hence $R/(J(T)\cap R)\cong T/J(T)$, as rings, and therefore $R/(J(T)\cap R)$ is a $J$-semisimple ring. This means that $J(T)\cap R$ is an intersection of a family of maximal left ideals of $R$ and therefore $J(R)\subseteq J(T)\cap R$. Thus $J(R)=J(T)\cap R$, by the previous remark. $(3)$ Assume that for each left primitive ideal $Q$ of $T$, $Q$ is not contained in $R$. Therefore by $(1)$, we deduce that $Q$ is a maximal ideal of $T$ and $Q\cap R$ is a maximal ideal of $R$. Hence $J(R)\subseteq Q\cap R$, for each left primitive ideal $Q$ of $T$ and therefore $J(R)\subseteq J(T)\cap R$. Hence we are done by the previous remark. The proofs of $(4)-(7)$ are similar to $(2)$ and $(3)$. $(8)$ is an immediate consequence of $(1)$ and the fact that $R$ is a left Artinian ring. $(9)$ Assume that $K$ be a nonzero left $R$-submodule of $T$ and $0\neq x\in K$. Since $R$ is a left Artinian ring, we infer that $Rx$ is an Artinian submodule of $K$. It is clear that $Rx$, and therefore $K$, contains a simple (left) $R$-submodule. Thus $Soc({}_RT)$ is an essential left $R$-submodule of $T$. The final part is evident, for $Soc({}_RI)\neq 0$, and use Proposition \ref{t23}.
\end{proof}

We present the next lemma in more general setting (i.e., without assuming $R$ is a left Artinian ring).

\begin{lem}\label{t31}
Let $R$ be a maximal subring of a ring $R$ and $M, N$ are distinct maximal ideals of $T$ which are not contained in $R$. If $P:=M\cap R=N\cap R$, then the following hold:
\begin{enumerate}
\item $P=(R:T)=M\cap N$ is a maximal ideal of $R$. Moreover, $P$ is not a prime ideal of $T$.
\item $T/M\cong R/P\cong T/N$ as rings and left/right $R$-modules.
\item $|Max(T)|\leq 1+|Max(R)|$
\item If $R$ is a left Noetherian (resp. Artinian) ring, then $T$ is a left Noetherian (resp. Artinian) $R$-module; in particular $T$ is a left Noetherian (resp. Artinian) ring.
\end{enumerate}
\end{lem}
\begin{proof}
$(1)$ and $(2)$, since $M$ is not contained in $R$, we deduce that $R+M=T$ and therefore $R/(M\cap R)\cong T/M$ as rings and also left/right $R$-modules. Thus $P=M\cap R=N\cap R$ is a maximal ideal of $R$. Clearly, $P=(M\cap N)\cap R$. If $M\cap N$ is not contained in $R$, then by maximality of $R$ we conclude that $R+(M\cap N)=T$ and therefore $R/P\cong T/(M\cap N)$, as rings, which is impossible, for $R/P$ is a simple ring but $T/(M\cap N)$ is not. Hence $M\cap N\subseteq R$ and therefore $P=M\cap N\subseteq (R:T)$  is a maximal ideal of $R$. Thus $P=M\cap N=(R:T)$. For the final part of $(1)$, if $(R:T)$ is a prime ideal of $T$, then $MN\subseteq (R:T)$, implies that $M\subseteq (R:T)\subseteq R$ or $N\subseteq (R:T)\subseteq R$, which are impossible. Hence $(R:T)$ is not a prime ideal of $T$. $(3)$ Let $Q$ be a prime ideal of $T$ such that $Q\cap R=P$. Hence $M\cap N\subseteq Q$ and therefore $M\subseteq Q$ or $N\subseteq Q$, for $Q$ is prime. Thus $Q=M$ or $Q=N$. Therefore for each maximal ideal $Q$ of $T$, if $Q\neq M, N$, then we conclude that $Q\cap R\neq P$. Now if $Q, Q'$ are maximal ideal of $T$ which are not equal to $M$ or $N$, then $Q\cap R\neq Q'\cap R$, for otherwise by $(1)$ $M\cap N=Q\cap Q'\subseteq Q$, i.e., $Q=M$ or $Q=N$, which is absurd. Also note that by $(1)$, $Q$ is not contained in $R$, and therefore similar to the proof of $(1)$, we deduce that $Q\cap R$ is a maximal ideal of $R$. Hence $M$ and $N$ have a same contraction to $R$ and other maximal ideals of $T$ have different contractions (which also are not equal to $P$), therefore $|Max(T)|\leq 1+|Max(R)|$. Finally for $(4)$, note that $T/(M\cap N)$ embeds in $T/M\times T/N$ as left $T$-modules and therefore as left $R$-modules too. Hence by $(2)$, $T/(M\cap N)$ embeds in $R/P\times R/P$ as left $R$-modules. Thus, if $R$ is a left Noetherian (resp. Artinian) ring, then $T/(M\cap N)$ is a left Noetherian (resp. Artinian) $R$-module and since $M\cap N\subseteq R$, we immediately conclude that $T$ is a left Noetherian (resp. Artinian) $R$-module. The final part is evident.
\end{proof}

Assume that $R$ be an Artinian maximal subring of a ring $T$. For the next observations in this section we need that following remark about the behaviour of conductor ideals $(R:T)_l$ and $(R:T)_r$ in $T$.    

\begin{rem}\label{t32}
Let $R$ be a maximal subring of a ring $T$ and $R$ be a left Artinian ring. Then one of the following holds:
\begin{enumerate}
\item $(R:T)_l$ is not an ideal of $T$. In this case $(R:T)_r$ is not an ideal of $T$ and $(R:T)_l\neq (R:T)_r$. In particular, $R=\mathbb{I}((R:T)_l)=\mathbb{I}((R:T)_r)$.
\item $(R:T)=(R:T)_l=(R:T)_r$.
\end{enumerate}
To see this, first note that $(R:T)\subseteq (R:T)_l, (R:T)_r$ and $(R:T)_l$, $(R:T)_r$ are prime ideals of $R$ and therefore are maximal, for $R$ is a left Artinian ring. Also note that $(R:T)_l$ is an ideal of $T$ if and only if $(R:T)_l=(R:T)$. Since $(R:T)\subseteq (R:T)_r$ and $(R:T)_l$ is a maximal ideal of $R$, we infer that $(R:T)_l$ is an ideal of $T$ if and only if $(R:T)_l=(R:T)_r$, i.e., $(R:T)_l=(R:T)=(R:T)_r$. Similarly, $(R:T)_r$ is an ideal of $T$ if and only if $(R:T)_l=(R:T)=(R:T)_r$. Thus the first part of $(1)$ and $(2)$ are clear now. Finally, assume that $(R:T)_l$ and therefore $(R:T)_r$ are not ideals of $T$. Since these are ideals of $R$ and $R$ is a maximal subring of $T$, we immediately conclude that $R=\mathbb{I}((R:T)_l)=\mathbb{I}((R:T)_r)$.
\end{rem}

Also note that in Case $(1)$ of the previous remark, either $(R:T)_lT=T$ or $(R:T)_lT$ is a proper ideal of $T$. In the following, we study these cases separately, to find some relations between Jacobson radicals of the ring extension $R\subseteq T$, where $R$ is a left Artinian ring and is a maximal subring of $T$.

\begin{thm}\label{t33}
Let $R$ be a maximal subring of a ring $T$ and $R$ be a left Artinian ring. Assume that $(R:T)_lT=T$, then ${}_RT$ and $((R:T)_l)_R$ are finitely generated. In particular, $T$ is a left Artinian ring. Moreover, either $J(R)=J(T)\cap R$ or the following hold:
\begin{enumerate}
\item $Q:=(R:T)\in Max(T)$. In particular, $J(T)\subseteq J(R)\cap (R:T)$.
\item $T(R:T)_r=T$. In particular, $T_R$ is finitely generated.
\item If $M$ and $N$ are distinct maximal ideals of $T$ and $M,N\neq Q$, then $M\cap R\neq N\cap R$. In particular, $Max(R)=\{(R:T)_l, (R:T)_r\}\cup\{P\cap R\ |\ P\in Max(T), P\neq Q\}$ and $|Max(R)|=|Max(T)|+1$.
\end{enumerate}
\end{thm}
\begin{proof}
Since $(R:T)_lT=T$, there exist $a_1,\ldots, a_n\in (R:T)_l$ and $t_1,\ldots, t_n\in T$ such that $a_1t_1+\cdots+a_nt_n=1$. Therefore for each $x\in T$ and $p\in(R:T)_l$, we have $x=xa_1t_1+\cdots+xa_nt_n$ and $a_1t_1p+\cdots+a_nt_np=p$. Since $p, a_i\in (R:T)_l$, we infer that $xa_i, t_ip\in R$. Therefore $T=Rt_1+\cdots+Rt_n$ and $(R:T)_l=a_1R+\cdots+a_nR$. Hence ${}_RT$ and $((R:T)_l)_R$ are finitely generated. It is clear that $T$ is a left Artinian $R$-module and therefore is a left Artinian ring. Now by $(3)$ of Lemma \ref{t30}, we have two cases: $(i)$ if $R$ does not contain any maximal ideal of $T$ (note $T$ is a left Artinian ring), then $J(R)=J(T)\cap R$. $(ii)$ hence assume that $R$ contains a maximal ideal $Q$ of $T$, which clearly is unique. Therefore $Q=(R:T)$ and thus $J(T)\subseteq (R:T)$ is contained in $R$. Hence $J(T)\subseteq J(R)$, by Remark \ref{t29}. Thus $(1)$ holds. For $(2)$, since $(R:T)_l$ is not an ideal of $T$, we deduce that $(R:T)_r$ is not an ideal of $T$, by Remark \ref{t32}. Hence $Q\subsetneq (R:T)_r\subsetneq T(R:T)_r$ and therefore by maximality of $Q$ we conclude that $T(R:T)_r=T$. Similar to the proof of ${}_RT$ is finitely generated, we deduce that $T_R$ is finitely generated. Finally for $(3)$, assume that $M$ and $N$ are distinct maximal ideals of $T$ and $M,N\neq Q$. If $M\cap R=N\cap R$, then by the first part of $(1)$ of Lemma \ref{t31}, we deduce that $Q=(R:T)=M\cap N$, which is absurd, by the second part of $(1)$ of Lemma \ref{t31}. Thus $M\cap R\neq N\cap R$. Also note that, since $M\neq Q$, we have $M\nsubseteq R$ and therefore $R+M=T$, for $R$ is a maximal subring of $T$. Hence $R/(R\cap M)\cong T/M$ as rings and therefore $R\cap M$ is a maximal ideal of $R$. Also note that since $(R:T)_lT=T$ and $T(R:T)_r=T$, we immediately conclude that no maximal ideals of $T$ contains $(R:T)_l$ or $(R:T)_r$. Hence, if $M_1=Q, M_2,\ldots, M_k$ denote all maximal ideals of $T$ (note, $T$ is a left Artinian ring and therefore $Max(T)$ is finite), then $(R:T)_l, (R:T)_r, M_2\cap R,\ldots, M_k\cap R$ are maximal ideals of $R$. Conversely, if $P$ is a maximal ideal of $R$, then there exists a maximal ideal $Q'$ of $T$ such that $Q'\cap R\subseteq R$ (note $R\setminus P$ is a $m$-system in $R$ and therefore in $T$ and note that $T$ is a left Artinian ring). Now we have two cases: $(a)$ $Q'=Q$ and thus $Q\subseteq P$. Since $(R:T)_l(R:T)_r\subseteq Q\subseteq P$, we deduce that $P=(R:T)_l$ or $P=(R:T)_r$, for $P$ is a prime ideal and $R$ is a left Artinian ring. $(b)$ If $Q'\neq Q$, then as we see earlier $Q'\cap R$ is a maximal ideal of $R$ and therefore $P=Q'\cap R$. Hence $Max(R)=\{(R:T)_l, (R:T)_r\}\cup \{M_2\cap R,\ldots, M_k\cap R\}$ and therefore $|Max(R)|=|Max(T)|+1$.
\end{proof}

In the next result we assume that $(R:T)_l$ is not an ideal of $T$ and $(R:T)_lT$ is a proper ideal of $T$, where $R$ is an Artinian maximal subring of a ring $T$.

\begin{thm}\label{t34}
Let $R$ be a maximal subring of a ring $T$ and $R$ be a left Artinian ring. Assume that $(R:T)_l$ is not an ideal of $T$ and $(R:T)_lT$ is a proper ideal in $T$. Then the following hold:
\begin{enumerate}
\item $(R:T)_lT\in Max(T)$.
\item $R$ contains no maximal ideal of $T$; the contraction of distinct maximal ideals of $T$ are distinct maximal ideals of $R$, in particular $|Max(T)|\leq |Max(R)|$.
\item If $I:=l.ann_T((R:T)_l)$, then either $(R:T)_lT$ is a minimal prime ideal of $T$ and $J(R)=J(T)\cap R$ or $I^2=0$, $I\subseteq Nil_*(T)$ and $(R:T)_lT=r.ann_T(I)$, whenever $I\neq 0$.
\item Either $J(R)=J(T)\cap R$ or $(R:T)_l$ is not a finitely generated right ideal of $R$.
\item Either $T(R:T)_r\in Max(T)$ or there exists a natural number $n$ such that $T$ embeds in $((R:T)_r)^n$ as a right $T$-module and therefore in $R^n$ as a right $R$-module.
\end{enumerate}
\end{thm}
\begin{proof}
First note that by the assumptions $(R:T)\subsetneq (R:T)_l\subsetneq (R:T)_lT\subsetneq T$. Hence $R+(R:T)_lT=T$, by the maximality of $R$. Clearly $(R:T)_lT\cap R=(R:T)_l$ which is a maximal ideal of $R$, for $R$ is a left Artinian ring and $(R:T)_l$ is a prime ideal of $R$, by Lemma \ref{pt3}. Hence $R/(R:T)_l=R/((R:T)_lT\cap R)\cong T/(R:T)_lT$, as rings and thus $(R:T)_lT$ is a maximal ideal of $T$. Thus $(1)$ holds. For $(2)$, if $K$ is an ideal of $T$ which is contained in $R$, then $K\subseteq (R:T)\subsetneq (R:T)_lT$ and therefore $K$ is not a maximal ideal of $T$. Hence if $M$ is a maximal ideal of $T$, then $R+M=T$ by maximality of $R$, and therefore $R/(R\cap M)\cong T/M$ as rings, thus $R\cap M$ is a maximal ideal of $R$. Now assume that $M$ and $N$ are maximal ideals of $T$ with $R\cap N=R\cap M$. Thus by $(1)$ of Lemma \ref{t31}, we deduce that $M\cap N=(R:T)\subseteq (R:T)_lT$. Hence by $(1)$, we conclude that $M=(R:T)_lT$ or $N=(R:T)_lT$. In any cases, since $M\cap R=N\cap R$, we deduce that $(R:T)_l\subseteq M, N$ and thus $M=N=(R:T)_lT$. Therefore $(2)$ holds. For $(3)$, first suppose that $(R:T)_lT$ is a minimal prime ideal of $T$, then by $(3)$ of Lemma \ref{t30}, if $Q$ is a left primitive ideal of $T$ and $Q\subseteq R$, then $Q\subseteq (R:T)\subsetneq (R:T)_l\subsetneq (R:T)_lT$, which is absurd for $(R:T)_lT$ is a minimal prime ideal of $T$. Thus $R$ does not contain any left primitive ideal of $T$ and therefore $J(R)=J(T)\cap R$ by $(3)$ of Lemma \ref{t30}. Now assume that $(R:T)_lT$ is not a minimal prime ideal of $T$. For our claim we may suppose that $I\neq 0$. Since $(R:T)_lT$ is not a minimal prime ideal of $R$, we conclude that there exists a prime ideal $Q$ of $T$ which is properly contained in $P:=(R:T)_lT$. By $I(R:T)_l=0\subseteq Q$, we deduce that $I\subseteq Q\subsetneq P$. In particular, $I^2=0$ and therefore $I\subseteq Nil_*(T)$. Now note that since $I=l.ann_T((R:T)_lT)$ and $(R:T)_lT$ is a maximal ideal of $T$, we obtain that $(R:T)_lT=r.ann_T(I)$. Hence $(3)$ holds. For $(4)$, assume that $J(R)\neq J(T)\cap R$, then by $(3)$ of Lemma \ref{t30}, let $Q$ be a left primitive ideal of $T$ which is contained in $R$. Hence $Q\subseteq (R:T)\subsetneq (R:T)_l\subsetneq (R:T)_lT$. This immediately implies that $l.ann_T((R:T)_l)\subseteq Q$ and therefore $l.ann_T((R:T)_l)\subseteq R$ is a left Artinian $R$-module. Now suppose that $(R:T)_l$ is a finitely generated as a right ideal of $R$. Hence $(R:T)_l=a_1R+\cdots+a_nR$, for some $a_1,\ldots,a_n\in (R:T)_l$. Therefore $l.ann_T((R:T)_l)=l.ann_T(a_1)\cap\cdots\cap l.ann_T(a_n)$, which immediately implies that $T/l.ann_T((R:T)_l)$ can be embedded in $T/l.ann_T(a_1)\times\cdots\times T/l.ann_T(a_n)$, as left $T$-modules. Since $T/l.ann_T(a_i)\cong Ta_i$ as left $T$-modules and $a_i\in (R:T)_l$, we conclude that $T/l.ann_T((R:T)_l)$ embeds in $((R:T)_l)^n$ as left $T$-module and therefore as left $R$-module too. This immediately implies that $T/l.ann_T((R:T)_l)$ is a left Artinian $R$-module and therefore $T$ is a left Artinian ring. Hence $Q$ is a maximal ideal of $T$, which is absurd. Thus $(R:T)_l$ is not a finitely generated right ideal of $T$. For $(5)$, note that either $T(R:T)_r$ is a proper ideal of $T$ or $T(R:T)_r=T$. First assume that $T(R:T)_r$ is a proper ideal of $T$. Then clearly, $T(R:T)_r\cap R=(R:T)_r$, for $(R:T)_r$ is a maximal ideal of $R$, for $R$ is a left Artinian ring and $(R:T)_r$ is a prime ideal of $R$ by Lemma \ref{pt3}. By Remark \ref{t32}, we deduce that $T(R:T)_r$ is not contained in $R$, for $(R:T)_l$ and $(R:T)_r$ are not ideals of $T$. Hence $R+T(R:T)_r=T$, by maximality of $R$ and therefore $R/(R:T)_r\cong T/T(R:T)_r$ as rings, which immediately implies that $T(R:T)_r$ is a maximal ideal of $T$. Now assume that $T(R:T)_r=T$, in this case similar to the proof of $(4)$ and the proof of Theorem \ref{t33}, one can complete the proof.
\end{proof}

By $(4)$ of the previous theorem we have the following immediate corollary.

\begin{cor}
Let $R$ be a maximal subring of a ring $T$ and $R$ be an Artinian ring. Assume that $(R:T)_l$ is not an ideal of $T$ and $(R:T)_lT$ is a proper ideal in $T$. Then $J(R)=J(T)\cap R$
\end{cor}

Finally we consider case $(2)$ of Remark \ref{t32}, in the next two results.

\begin{prop}\label{t35}
Let $R$ be a maximal subring of a ring $T$. Assume that $(R:T)_l=(R:T)=(R:T)_r$, $R$ is a left Artinian ring and $T$ is a zero-dimensional ring. Then $J(R)=J(T)\cap R$.
\end{prop}
\begin{proof}
First note that by Remark \ref{t29}, $J(T)\cap R\subseteq J(R)$. By $(3)$ of Lemma \ref{t30}, we have two cases: $(i)$ if each left primitive ideal of $T$ is not contained in $R$, then $J(R)=J(T)\cap R$ and hence we are done. $(ii)$ there exists a left primitive ideal $Q$ of $T$ such that $Q\subseteq R$. Hence by our assumption $Q$ is a maximal ideal of $T$ and therefore $Q=(R:T)$. Therefore in this case we infer that $J(T)\subseteq J(R)\subseteq Q$ (note, $Q=(R:T)_l=(R:T)=(R:T)_r$ is a prime ideal of $R$, by Lemma \ref{pt3}, and $R$ is a left Artinian ring). Now let $P$ be a left primitive ideal of $T$ and $P\neq Q$. Therefore $P\nsubseteq R$ and thus $R+P=T$, by maximality of $R$. Thus $R/(R\cap P)\cong T/P$ as rings, and therefore $P\cap R$ is a maximal ideal of $R$ (note that $T$ is zero-dimensional and hence $P$ is maximal). Thus $J(R)\subseteq P\cap R$. Hence for each left primitive ideal $P$ of $R$ we have $J(R)\subseteq P$ and therefore $J(R)\subseteq J(T)$. Thus $J(R)=J(T)$.
\end{proof}

We conclude this paper by the following main result.

\begin{thm}\label{t36}
Let $R$ be a maximal subring of a ring $T$ and $R$ be a left Artinian ring. Assume that $(R:T)_l=(R:T)=(R:T)_r$. Then $T/(R:T)\cong \mathbb{M}_n(S)$, for some ring $S$ and a natural number $n$, and $S$ has a maximal subring which is a division ring. Moreover, either $T$ is a left Artinian ring (in particular, $J(R)=J(T)\cap R$) or $S$ (and also $T/(R:T)$) has at most one nonzero proper ideal, and one of the following holds:
\begin{enumerate}
\item $S$ is a simple ring, i.e., $(R:T)$ is a maximal ideal of $T$.
\item $S$ has a unique nonzero proper ideal $M$ such that $M^2=0$. In particular, $J(S)=M$.
\item $S$ has a unique nonzero proper ideal $M$ such that $M^2=M$. In particular, either $S$ is a primitive ring or $J(S)=M$.
\end{enumerate}
Moreover, if $Q$ is the unique maximal ideal of $T$ which contains $(R:T)$ in items $(1)-(3)$, then the following hold:
\begin{itemize}
\item[$(a)$] $Max(T)=\{Q\}\cup \{N\in Spec(T)\ |\ N\nsubseteq R\}$.
\item[$(b)$] each prime ideal of $T$ is either maximal or is contained in $Q$. In particular, $dim(T)=ht(Q)$.
\item[$(c)$] if $N_1$ and $N_2$ are maximal ideals of $T$ and $N_i\neq Q$, then $N_1\cap R\neq N_2\cap R$.
\item[$(d)$] $|Max(T)|\leq |Max(R)|$.
\end{itemize}
\end{thm}
\begin{proof}
Let $P:=(R:T)$, then by our assumption we deduce that $P$ is a maximal ideal of $R$ and therefore $R/P$ is of the form $\mathbb{M}_n(D)$ for some natural number $n$ and a division ring $D$. Now since $R/P$ is a maximal subring of $T/P$, we immediately conclude that $T/P$ is of the form $\mathbb{M}_n(S)$ for a ring $S$, where $D$ is a maximal subring of $S$. Since $R/P$ is a simple ring, by Theorem \ref{pt2} we have two cases: $(i)$ $T/P$ has exactly two nonzero proper (maximal) ideals $M/P$ and $N/P$ such that $(M/P)\cap (N/P)=0$, i.e., $M\cap N=P$ and therefore by $(3)$ of Lemma \ref{t31}, we deduce that $T$ is a left Artinian ring (and hence $J(R)=J(T)\cap R$, by the previous proposition). $(ii)$ one of the items $(1), (3), (4)$ of Theorem \ref{pt2} holds. Now one can easily see that in case $(1)$ of Theorem \ref{pt2}, we obtain $(1)$, in case $(4)$ of Theorem \ref{pt2}, we deduce $(2)$, and in case $(3)$ of Theorem \ref{pt2}, we have $(3)$. Now assume that $Q$ is the unique maximal ideal of $T$ which contains $(R:T)$, (i.e., in case $(1)$, $Q=(R:T)$ and otherwise $Q=\mathbb{M}_n(M)$ in cases $(2)$ and $(3)$). $(a)$ if $P$ is a prime ideal of $T$ which is not contained in $R$, then by maximality of $R$ we conclude that $R+P=T$ and therefore $R/(R\cap P)\cong T/P$ as rings. Thus $R\cap P$ is a prime ideal of $R$, and hence is a maximal ideal of $R$, for $R$ is a left Artinian ring. Hence $P$ is a maximal ideal of $T$. Conversely, if $N$ is a maximal ideal of $T$ and $N\neq Q$, then clearly, $N\nsubseteq R$, for otherwise we conclude that $N\subseteq (R:T)\subseteq Q$ which is absurd. Thus $(a)$ holds. For $(b)$, assume that $P$ is a prime ideal of $T$, if $P$ is not contained in $R$, as we see in $(a)$, $P$ is a maximal ideal of $T$, otherwise $P\subseteq R$ and therefore $P\subseteq (R:T)\subseteq Q$. $(3)$ is evident by our assumption in these cases and by $(1)$ of Lemma \ref{t31}. $(d)$ is a consequence of $(a)$, $(c)$, our assumption in these cases and by Lemma \ref{t31}, in fact the function $N\longmapsto N\cap R$ is a one-one function from $Max(T)$ into $Max(R)$.
\end{proof}

\centerline{\Large{\bf Acknowledgement}}
The author is grateful to the Research Council of Shahid Chamran University of Ahvaz (Ahvaz-Iran) for
financial support (Grant Number: SCU.MM1403.721)

% ----------------------------------------------------------------
%\bibliographystyle{amsplain}
%\bibliography{}

\begin{thebibliography}{mm}
\bibitem{andcam}
D.D. Anderson, V. Camillo, Semigroups and rings whose zero products commute, Comm. Algebra, {\bf 27} (6) (1999) 2847-2852.

\bibitem{azq}
A. Azarang, Non-quasi duo rings have maximal subrings, Int. Math. Forum, {\bf 5} (20) (2010) 979-994.


\bibitem{azcond}
A. Azarang, The conductor ideals of maximal subrings in non-commutative rings, preprint (submitted) (2024) Arxiv: \url{https://arxiv.org/abs/2406.12890}.

\bibitem{azconch}
A. Azarang, Conch maximal subrings, Comm. Algebra, {\bf 50} (3) (2022) 1267-1282.

\bibitem{azkra}
A. Azarang, O.A.S. Karamzadeh, On the existence of maximal subrings in commutative artinian rings, J. Algebra Appl. {\bf 9} (5) (2010) 771-778.

\bibitem{azkrc}
A. Azarang, O.A.S. Karamzadeh, On maximal subrings of commutative rings, Algebra Colloq. {\bf 19} (Spec 1) (2012) 1125-1138.

\bibitem{azkrm}
A. Azarang, O.A.S. Karamzadeh, Most commutative rings have maximal subrings, Algebra Colloq. {\bf 19} (Spec 1) (2012) 1139-1154.


\bibitem{akn}
A. Azarang, O.A.S. Karamzadeh, A. Namazi, Hereditary properties between a ring and its maximal subrings, Ukrainian. Math. J. {\bf 65} (7) (2013) 981-994.

\bibitem{adjex}
L.I. Dechene, Adjacent extensions of rings, Ph.D. Dissertation, University of California, Riverside, (1978).

\bibitem{dbsid}
D.E. Dobbs, J. Shapiro, A classification of the minimal ring extensions of an integral domain, J. Algebra 305 (2006) 185-193.

\bibitem{dbsc}
D.E. Dobbs, J. Shapiro, A classification of the minimal ring extensions of certain commutative rings, J. Algebra 308 (2007) 800-821.


\bibitem{dorsy}
T.J. Dorsey, Z. Mesyan, On minimal extension of rings, Comm. Algebra, {\bf 37} (10) (2009) 3463-3486.

\bibitem{frd}
D. Ferrand, J.-P. Olivier, Homomorphismes minimaux d’anneaux, J. Algebra {\bf 16} (1970) 461-471.

\bibitem{glf}
R. Gilmer, Some finiteness conditions on the set of overring an integral domain. Proc. Amer. Math. Soc. {\bf 131} (8) (2003) 2337-2346.

\bibitem{ghi}
R. Gilmer, W.J. Heinzer, Intersection of quotient rings of an integral domain. J. Math Kyoto Univ. {\bf 7} (2) (1967) 133-150.


\bibitem{good}
K.R. Goodearl, R.B. Warfield, JR, {\it An introduction to noncommutative Noetherian rings}, Camberidge University Press, Second Edition, (2004).


\bibitem{krmnaz}
O.A.S. Karamzadeh, N. Nazari, On maximal commutative subrings of non-commutative rings. Comm. Algebra, {\bf 46} (12) (2018) 5083-5115.

\bibitem{klein}
A.A. Klein, The finiteness of a ring with a finite maximal subrings. Comm. Algebra, {\bf 21} (4) (1993) 1389-1392.

\bibitem{laffey}
T.J. Laffey, A finiteness theorem for rings, Proc. R. Ir. Acad. {\bf 92} (2) (1992) 285-288.

\bibitem{lam}
T.Y. Lam, {\it A first course in noncmmutative rings}, Second Edition, Springer-Verlag, (2001).

\bibitem{lam2}
T.Y. Lam, {\it Lectures on Modules and Rings}, Springer-Verlag, (1999).


\bibitem{lamq}
T.Y. Lam, A.S. Dugas, Quasi-duo rings and stable range descent, J. Pure Appl. Algebra {\bf 195} (2005) 243-259.



\bibitem{macrob}
J.C. MacConnell, J.C. Robson, {\it Noncommutative Noetherian rings}, New York, Wilely-Interscience, (1987).


\bibitem{modica}
M.L. Modica, Maximal subrings, Ph.D. Dissertation, University of Chicago, (1975).

\bibitem{mscnfsr}
G. Peruginelli, N.J. Werner, Maximal subrings and covering numbers of finite semisimple rings, Comm. Algebra \textbf{46} (11) (2018) 4724-4738.


\bibitem{abmin}
G. Picavet and M. Picavet-L'Hermitte. About minimal morphisms, in: Multiplicative ideal theory in commutative algebra, Springer-Verlag, New York, (2006) pp. 369-386.


\bibitem{rvn}
L. Rowen, {\it Ring theory}, Vol. I, Pure Appl. Math., 127 Academic Press, Inc., Boston, MA, (1988).


\end{thebibliography}

\end{document}